%% file: Weak_error_skew_diff.tex
\documentclass{ps}
\usepackage{amsmath,amsfonts}
\usepackage{color}
\usepackage[applemac]{inputenc}
\usepackage{mathrsfs}
\usepackage{mathtools}

\def \E{\mathbb{E}}
\def \R{\mathbb{R}}

\def\leftB{[\![}
\def\rightB{]\!]}

\include{macro}

\def \R{\mathbb{R}}
\def \N{\mathbb{N}}
\def \E{\mathbb{E}}

\newtheorem{thm}{Theorem}[section]

\newtheorem{lem}[thm]{Lemma}

\begin{document}
\title{On the weak approximation of a skew diffusion by an Euler-type scheme}
\author{N. Frikha}\address{LPMA, Universit\'e Paris Diderot, 5 rue Thomas Mann 75013 Paris, email: frikha@math.univ-paris-diderot.fr}
\date{\today}
\begin{abstract} We study the weak approximation error of a skew diffusion with bounded measurable drift and H\"older diffusion coefficient by an Euler-type scheme, which consists of iteratively simulating skew Brownian motions with constant drift. We first establish two sided Gaussian bounds for the density of this approximation scheme. Then, a bound for the difference between the densities of the skew diffusion and its Euler approximation is obtained. Notably, the weak approximation error is shown to be of order $h^{\eta/2}$, where $h$ is the time step of the scheme, $\eta$ being the H\"older exponent of the diffusion coefficient. 
\end{abstract}
\maketitle

\section{Introduction}
\subsection{Statement of the Problem}
\label{stat:problem:sec}
We consider the unique weak solution of the following $\R$-valued stochastic differential equation (SDE) $(X_t)_{t\geq0}$ with dynamics
\begin{equation}\label{skew:diff}
X_t = x + \int_0^t b(X_s) ds + \int_0^t \sigma(X_s) dW_s + (2\alpha-1) L^{0}_t(X)
\end{equation}

\noindent where $W=(W_t)_{t\geq0}$ is a one dimensional Brownian motion defined on a filtered probability space $(\Omega, \mathcal{F}, (\mathcal{F}_t)_{t\geq0},\mathbb{P})$ satisfying the usual assumptions and $L^{0}(X)$ is the symmetric local time of $X$ at the origin.

When $b=0$ and $\sigma=1$, the solution to \eqref{skew:diff} is called the skew Brownian motion. Harrison and Shepp \cite{Har:Shepp} proved that if $|2\alpha-1|\leq1$ then there is a unique strong solution and if $|2\alpha-1|>1$, there is no solution. The case $\alpha=1$ corresponds to reflecting Brownian motion.

Here we will assume that $\alpha \in (0,1)$, $b$ is measurable, bounded, $\sigma$ is uniformly elliptic, bounded and $a=\sigma^2$ is  $\eta$-H\"older continuous for some $\eta \in (0,1]$. The previous assumptions guarantee the existence of a unique weak solution to \eqref{skew:diff}. Moreover, for any $(t,x)\in \R^{*}_+ \times \R$, $X_t$ admits a density $y\mapsto p(0,t,x,y)$, which is continuous on $\R^{*}$ and satisfies a Gaussian upper-bound. We refer to \cite{kohatsu:taguchi:zhong} for more details, see also \cite{Legall}. We also refer the interested reader to the recent survey \cite{Lejay} and the references therein for various applications of such equation.

As far as numerical approximation is concerned several discretization schemes of \eqref{skew:diff} have been proposed. For instance, Lejay and Martinez \cite{lejay:mart} recently introduced a numerical scheme based on the simulation of a skew Brownian motion. Martinez and Talay \cite{Mart:Talay} proposed a transformed Euler discretization scheme of an equation similar to \eqref{skew:diff} to approximate the solution of a linear parabolic diffraction problem and provide a weak convergence rate. Another approximation scheme based on random walk has also been studied by \'{E}tor \cite{Etore}. In \cite{etore:martinez}, \'{E}tor and Martinez developed a simulation scheme for skew diffusions with constant diffusion coefficient.
 
 To approximate equation \eqref{skew:diff} on the time interval $[0,T]$, $T>0$, we introduce the Euler scheme $(X^{N}_t)_{t\in [0,T]}$ associated to the time step $h=T/N$, $N\in \N^*$ and time grid $t_i = i h $, $i\in \leftB0,N\rightB$, defined by $X^{N}_0 = x$ and for all $t_i \leq t \leq t_{i+1}$ 
 \begin{equation}\label{Euler:scheme}
 X^{N}_t = X^{N}_{t_i} +  b(X^{N}_{t_i}) (t-t_i) + \sigma(X^{N}_{t_i}) (W_t-W_{t_i}) + (2\alpha-1) L^{0}_{t-t_i}(X^{N}(t_i,X^{N}_{t_i})).
 \end{equation}
 
 Let us note that the above scheme does not correspond to a standard Euler-Maruyama approximation scheme since we do not discretize the part corresponding to the local time in \eqref{stat:problem:sec}. However, its computation only requires to be able to simulate exactly the skew Brownian motion with a constant drift at time $t-t_i$. This process is known to be exactly simulatable and we refer to \cite{etore:martinez} for the exact expression of its density.  
 
Two kind of quantities are of interest when studying the weak approximation error of \eqref{skew:diff}. The first one concerns the quantity $\mathcal{E}^{1}_{\mathcal{W}}:= \E_x[f(X_T)]-\E_x[f(X^{N}_T)]$, where $f$ is a test function that lies in a suitable class. The second one writes $\mathcal{E}^{2}_{\mathcal{W}}:=(p-p_{N})(0,t,x,y)$, where $y\mapsto p(0,t,x,y)$ (resp. $y\mapsto p_{N}(0,t,x,y)$) denotes the density of the unique solution $X_t$ of \eqref{skew:diff} taken at time $t$ and starting from $x$ at time $0$ (resp. of $X^{N}_t$ given by the scheme \eqref{Euler:scheme} at time $t$ and starting from $x$ at time $0$) when it exists. The problem of interest is to give a bound or an expansion of these two quantities in terms of the discretization step $h$. 

Let us note that the two quantities $\mathcal{E}^{1}_{\mathcal{W}}$ and $\mathcal{E}^{2}_{\mathcal{W}}$ are of a different nature and require different techniques and methodology depending on the considered class of test functions. Indeed, in the case of SDE driven by a Brownian motion ($\alpha=1/2$ i.e. without local time), provided the test function $f$, and the coefficients $b$ and $\sigma$ are smooth and $f$ is of polynomial growth, Talay and Tubaro \cite{tala:tuba:90} derived an error expansion at order $1$ for $\mathcal{E}^{1}_{\mathcal{W}}$. The weak approximation error for Lvy driven SDEs has been studied in Jacod \& al. \cite{Jacod:kurtz:meleard:protter} under appropriate smoothness of coefficients and the test function $f$. In the case of SDEs driven by a Brownian motion, the same result may be extended to bounded Borel functions under a non-degeneracy assumption of H\"ormander type on the diffusion, see Bally and Talay \cite{ball:tala:96:1} or to the case of the density \cite{ball:tala:96:2}. In the uniformly elliptic setting, Konakov and Mammen \cite{kona:mamm:02}, in the Brownian case, and Konakov and Menozzi \cite{kona:meno:10}, for stable driven SDEs, successfully derived an expansion for $\mathcal{E}^{2}_{\mathcal{W}}$ in powers of $h$ by using a continuity approach known as the parametrix method as developed in McKean and Singer \cite{mcke:sing:67} and Friedmann \cite{friedman:64}. Roughly speaking, it consists in expanding the transition density of the initial equation around a process with frozen coefficients for which explicit expression of the density and its derivatives are available. This approach seems to be quite robust and useful since it can be applied in various contexts such that discrete Markov chains \cite{kona:mamm:00}. However, little attention has so far been given to the case of SDEs with non-smooth coefficients.

In the classical case of uniformly elliptic diffusion processes, we mention the work of Mikulevicius and Platen \cite{Miku:Platen} who established an error bound for the weak approximation error $\mathcal{E}^{1}_{\mathcal{W}}$ of order $h^{\eta/2}$ provided $f\in \mathcal{C}^{2+\eta}([0,T]\times \R^{d})$ and the coefficients $b$ and $\sigma$ are $\eta$-H\"{o}lder continuous in space and $\eta/2$-H\"{o}lder continuous in time. We also refer the reader to \cite{miku:2012} and \cite{miku:zhang:2015} for some recent extensions of this result to the case of Lvy driven SDEs. 
 
More recently, Konakov and Menozzi \cite{kona:koz:menozzi} also derived an upper bound for $\mathcal{E}^{2}_{\mathcal{W}}$ of order $h^{\frac{\eta}{2}-C\psi(h)}$, where $\psi(h)$ is a slowly varying factor that goes to zero as $h\rightarrow0$ under the assumption that the coefficients $b$ and $\sigma$ are H\"older continuous. Their strategy consists in introducing perturbed dynamics of the considered SDE and its scheme by suitably mollifying the coefficients of both dynamics and to quantify the distance between the densities and their respective perturbations. Let us mention that another perturbative approach has been considered in Kohatsu-Higa \& al. \cite{kohatsu:Lejay:Yasuda} for the weak approximation error $\mathcal{E}^{1}_{\mathcal{W}}$ with non-smooth coefficients. In order to establish an error bound between the transition densities of \eqref{skew:diff} and \eqref{Euler:scheme}, we rely on the parametrix methodology. Our approach allows to establish that this difference is of order $h^{\eta/2}$, without any additional varying factor (and time singularity) as in Theorem 1 in \cite{kona:koz:menozzi}, that is to extend to the case of skew diffusions the result in \cite{Miku:Platen} and to handle the densities itself (note again that $\alpha=1/2$ corresponds to the case of one dimensional time-homogeneous Brownian SDEs). One of the main advantages of the parametrix approach is the removal of the drift in the analysis of the approximation error. In particular, we remove the H\"{o}lder regularity assumption of the drift coefficient $b$ by allowing it to be only bounded measurable.

The paper is organized as follows, we first give our standing assumptions and notations in Section \ref{assumptions:notations:sec}. We state our main results in Section \ref{main:results}. Section \ref{proof:gauss:bound} is dedicated to the proofs of Aronson's estimates and the weak approximation error on the densities. The main tool for both results is a discrete parametrix representation of McKean-Singer type for the density of the scheme, see \cite{mcke:sing:67} and \cite{kona:mamm:00}, \cite{kona:mamm:02}. In Section \ref{Appendix} we prove some key technical lemmas that are used in our proofs.
 
\subsection{Assumptions and notations}\label{assumptions:notations:sec}
We here specify some assumptions on the coefficients $b$ and $\sigma$. 
 \begin{trivlist}
\item[\A{HR}] The drift $b$ is bounded measurable and $a=\sigma^2$ is $\eta$-H\"older continuous, for some $\eta \in (0,1]$. That is, there exists a positive constant $L$ such that
$$
\sup_{x\in \R} |b(x)| + \sup_{(x,y)\in \R^2, x\neq y} \frac{|a(x)-a(y)|}{|x-y|^{\eta}} < L.
$$

\item[\A{HE}] The diffusion coefficient is uniformly elliptic that is there exists $\lambda >1$ such that for every $x \in \R^2$, $\lambda^{-1} \leq a(x) < \lambda$. Since $\sigma$ is continuous, without loss of generality, we may assume that $\sigma$ is positive.
 \end{trivlist}

In the following we will denote by $C$ a generic positive constant that may depend on $b, \sigma, T$. We reserve the notation $c$ for constants depending on $\lambda, \eta, b, \sigma$ but not on $T$. Importantly, the constants $C,c$ are uniform with respect to the discretization time step $h$. Moreover, the value of both $C,c$ may eventually change from line to line. The notation $g_{C}$ stands for the Gaussian kernel, namely $g_{C}(y-x) := (1/(2\pi C))^{1/2} \exp(-(y-x)^2/(2C))$. We finally define the Mittag-Leffler function $E_{\alpha,\beta}(z) = \sum_{n\geq 0} z^{n}/\Gamma(\alpha n +\beta)$, $z\in \mathbb{R}$, $\alpha,\ \beta>0$.

\section{Main results}\label{main:results}
Before stating our results, let us first justify that under assumptions \A{HR} and \A{HE}, the random variable $X^{N}_{t_j}$, $j \in \leftB 1, N\rightB$, given by the discretization scheme \eqref{Euler:scheme} admits a positive density. For $x\in \R$, $0\leq j < i \leq N$ and any bounded Borel function $f$, one has
\begin{equation}
\label{exp:scheme}
\E[f(X^{N}_{t_i})| X^{N}_{t_j}=x] = \int_{\R^{i-j-1} \times \R} p_{N}(t_j, t_{j+1},x,y_{j+1}) \times \cdots \times p_{N}(t_{i-1},t_i,y_{i-1},y_i) f(y_i) dy_{j+1} \cdots dy_{i}
\end{equation}

\noindent where $p_{N}(t_k, t_{k+1},y_k,y_{k+1})$ stands for the density of the random variable $ X^{N}_{t_{k+1}} = y_k+  b(y_k) h + \sigma(y_k) (W_{t_{k+1}}-W_{t_k}) + (2\alpha-1) L^{0}_{h}(X^{N}(t_k, y_k))$, which satisfies $p_{N}(t_k, t_{k+1},y_k,y_{k+1})= q(h,y_k,y_{k+1})/\sigma(y_k)$, where $q(h,y_k,y_{k+1})$ stands for the density at time $h$ and terminal point $y_{k+1}$ of the skew Brownian motion starting from $y_k/\sigma(y_k)$ at time $0$ with constant drift part $b(y_k)/\sigma(y_k)$. Again, we refer to \cite{etore:martinez} for the exact expression of $q$. From equation \eqref{exp:scheme}, we clearly see that under assumptions \A{HR} and \A{HE}, the discretization scheme \eqref{Euler:scheme} admits a positive transition density, that we will denote $p_{N}(t_i,t_j,x,y)$, for any $0\leq t_i < t_j \leq T$, $(x,y) \in \R \times \R^{*}$. In particular, a Gaussian upper-bound has been established in \cite{kohatsu:taguchi:zhong} for the transition density of the skew diffusion \eqref{skew:diff} under \A{HR} and \A{HE}. Our first result is to prove similar Aronson's estimate, that is a Gaussian upper estimate but also a lower bound hold for the discretization scheme \eqref{Euler:scheme}.

\begin{THM}(Two sided Gaussian estimates for the scheme)\label{gauss:bound:scheme}
Under \A{HR} and \A{HE}, there exist two constants $C:=C(T,b,\sigma), c:=c(\lambda,\eta)>1$ such that for every $0 \leq j < i \leq N$,
$$
\forall (x,y) \in \R\times \R^{*},\ C^{-1} g_{c^{-1}(t_i-t_j)}(y-x) \leq p_{N}(t_j,t_i,x,y) \leq C g_{c(t_i-t_j)}(y-x).
$$

\end{THM}

We rely on a parametrix expansion of the density $p_{N}(t_j,t_i,x,y)$ to prove Theorem \ref{gauss:bound:scheme}. Such bounds were obtained in \cite{lem:men:2010} for the discretization schemes of uniformly elliptic diffusions and of some degenerate Kolmogorov processes which in turn allowed to derive concentration bounds for the statistical error in the Monte Carlo simulation method. The proof of Theorem \ref{gauss:bound:scheme} is postponed to Section \ref{proof:gauss:bound}.

Our second result concerns the weak approximation error $\mathcal{E}^{2}_{\mathcal{W}}= (p-p_{N})(0,t_i,x,y)$. Notably, we provide an error bound for the difference between the densities of the skew diffusion \eqref{skew:diff} and its approximation scheme \eqref{Euler:scheme}. Its proof is also postponed to Section \ref{proof:gauss:bound}.

\begin{THM}[Error bound on the difference between the densities]\label{weak:error:scheme} Assume that \A{HR} and \A{HE} hold. Then, there exists a constant $c:=c(\lambda,\eta)>1$ such that, for all $0\leq t_j < t_i \leq N$, one has 
$$
\forall (x,y) \in \times \R \times \R^{*}, \ |(p-p_{N})(t_j,t_i,x,y)| \leq C(T,b,\sigma) h^{\eta/2} g_{c (t_i - t_j)}(y-x)
$$

\noindent where $T\mapsto C(T,b,\sigma)$ is a non-decreasing positive function. 
\end{THM}

\begin{REM}Observe that the weak rate $h^{\eta/2}$ is coherent with previous results obtained in the literature for the weak approximation error $\mathcal{E}^1_{\mathcal{W}}$ of (not necessarily time homogeneous) SDE driven by a Brownian motion with H\"older coefficients, see e.g. \cite{Miku:Platen}. In \cite{kona:koz:menozzi}, the weak approximation error $\mathcal{E}^{2}_{\mathcal{W}} $ of the density is proved to be of order $h^{\frac{\eta}{2}-C\psi(h)}$, $\psi(h)= \log_3(h^{-1})/\log_2(h^{-1})$ when the coefficients $b$ and $\sigma$ are $\eta$-H\"older continuous in space and $\eta/2$-H\"older continuous in time. Moreover, in their results, the singularity in time is given by $C (t_i -t_j)^{-(1-\eta/2)\eta/2}$ whereas  this quantity does not appear in our result so it is tighter in this sense. For $\alpha=1/2$, which corresponds to the case of (time homogeneous) diffusion process (since the local time part vanishes), compared to \cite{kona:koz:menozzi}, our result notably removes the slowly varying factor $\psi(h)$ and shows that the drift plays no role in the approximation of the density since we only require $b$ to be a bounded measurable function. This phenomenon is not surprising since one of the advantages of the parametrix approach is the removal of the drift part in the analysis. Eventually, it should be possible to extend our strategy of proof to the case of multi-dimensional Brownian diffusion. 
\end{REM}

\section{Proofs of theorems \ref{gauss:bound:scheme} and \ref{weak:error:scheme}. } \label{proof:gauss:bound}
As already mentioned, the proof of Theorem \ref{gauss:bound:scheme} relies on a parametrix expansion of $p_{N}(t_j,t_i,x,y)$. As a warm-up, we briefly present to the reader the different steps to derive the parametrix expansion for \eqref{skew:diff} as developed in \cite{kohatsu:taguchi:zhong}. We refrain from discussing about the existence of the transition density for equation \eqref{skew:diff} but only present how to derive the expansion in an infinite series and the Gaussian upper-bound from it.

For $\alpha \in (0,1)$, we denote by $\mathcal{D}(\alpha)$ the set of bounded continuous functions $f:\R \rightarrow \R$ with bounded continuous derivatives $f'$ and $f''$ on $\R^{*}$ such that $f'(0+)$ and $f'(0-)$ exists and satisfies $\alpha  f'(0+) = (1-\alpha) f'(0-)$. Then, using the symmetric It\^o-Tanaka formula and the occupation time formula, one proves that the infinitesimal generator $\mathcal{L}$ of the Markov semigroup $(P_t)_{t\geq0}$ generated by \eqref{skew:diff} is given by
$$
\forall f \in D(\alpha), \ \forall x \in \R^{*}, \ \mathcal{L}f(x) = b(x) f'(x) + \frac{a(x)}{2} f''(x).
$$

Moreover, for any $f\in D(\alpha)$, one also obtains
\begin{equation}
\label{back:kolmog}
\forall (t,x) \in \R_+ \times \R^{*}, \ \frac{d P_t f}{dt}(x) = P_t \mathcal{L} f(x).
\end{equation}

We now define the frozen process that will be our main building block to construct the parametrix expansion for the transition density of \eqref{skew:diff}. For $z\in \R$, we consider the unique strong solution $X^{z}$ to the SDE
\begin{equation}
\label{frozen:proc}
X^{z}_t = x + \sigma(z) W_t + (2\alpha-1) L^{0}_t(X^{z}).
\end{equation}

Notice that compared to \eqref{skew:diff}, we froze the diffusion coefficient to $z$ and removed the drift since it will not play any significant role, as it will be clear later on. Its infinitesimal generator $\mathcal{L}^{z}$ writes
$$
\forall f \in D(\alpha), \ \forall x \in \R^{*}, \ \mathcal{L}^{z}f(x) = \frac{a(z)}{2} f''(x).
$$

Under \A{HE}, the transition density function of $(X^{z}_t)_{t\in [0,T]}$ exists and we denote it by $y\mapsto p^{z}(0,t,x,y)$. More precisely, one has to consider the two following cases:
\begin{itemize}
\item[Case 1:] For $x\geq0$, one has
$$
p^{z}(0,t,x,y)  = \left\{ g_{a(z)t}(y-x) + (2\alpha-1) g_{a(z)t}(y+x) \right\} \mbox{\textbf{1}}_{\left\{y \geq 0 \right\}}  + 2(1-\alpha) g_{a(z)t}(y-x)  \mbox{\textbf{1}}_{\left\{y < 0 \right\}}.
$$
\item[Case 2:] For $x<0$, one has
$$
p^{z}(0,t,x,y)  = \left\{ g_{a(z)t}(y-x) + (1-2\alpha) g_{a(z)t}(y+x) \right\} \mbox{\textbf{1}}_{\left\{y < 0 \right\}}  + 2 \alpha g_{a(z)t}(y-x)  \mbox{\textbf{1}}_{\left\{y \geq 0 \right\}}.
$$
\end{itemize}

Now, let $\varepsilon>0$. Noting that $x \mapsto p^{y}(0, \varepsilon,x,y) \in D(\alpha)$, from \eqref{back:kolmog}, for all $y\in \R^{*}$, we write
\begin{align}
\forall (x,y) \in \R \times \R^{*}, \ P_t p^{y}(0,\varepsilon,.,y)(x) - P^{y}_t p^{y}(0,\varepsilon,.,y)(x) & = \int_0^{t} ds \partial_s (P_s P^{y}_{t-s} p^{y}(0,\varepsilon,.,y))(x) \nonumber \\
 & = \int_0^t ds  \left(P_{s}\mathcal{L}P^{y}_{t-s} p^{y}(0,\varepsilon,.,y)(x)  - P_s \mathcal{L}^{y} P^{y}_{t-s} p^{y}(0,\varepsilon,.,y)(x) \right) \nonumber \\
 & = \int_0^t ds  P_{s}(\mathcal{L} - \mathcal{L}^{y}) p^{y}(0,\varepsilon+t-s,.,y)(x) \nonumber  \\
 & = p\otimes H(0,t+\varepsilon,x,y) \label{first:step:param}
\end{align}

\noindent where we used the Chapman-Kolmogorov equation $P^{y}_{t-s} p^{y}(0,\varepsilon,.,y)(x) = p^{y}(0,\varepsilon+t-s,x,y)$, introduced the notations $f\otimes g(s,t,x,y) = \int_s^t du \int_{\R} dz f(s,u,x,z) g(u,t,z,y)$ for the time space convolution and $H(s,t,x,y) = (\mathcal{L}-\mathcal{L}^y)p^{y}(s,t,x,y)$. From now on, we set $\tilde{p}(s,t,x,y) = p^{y}(s,t,x,y)$. In particular, the kernel $H$ writes
$$
H(s,t,x,y) = b(x) \partial_x \tilde{p}(s,t,x,y) + \frac{(a(x)-a(y))}{2} \partial^{2}_x \tilde{p}(s,t,x,y).
$$

Moreover, under assumptions \A{HR} and  \A{HE}, separating the four cases: $x,y\geq0$, $x\geq 0>y$, $y \geq 0>x$ and $x,y<0$, after some cumbersome but simple computations that we do not detail here (see Lemma 4.5 in \cite{kohatsu:taguchi:zhong} for more details), one gets
\begin{equation}
\label{kernel:bound}
\forall (x,y) \in \R \times \R^{*},Ê\ |H(s,t,x,y)| \leq \frac{C}{(t-s)^{1- \frac{\eta}{2}}} g_{c(t-s)}(y-x)
\end{equation}

\noindent with $C(T,b,\sigma, \eta) := C(\lambda, \eta) (|b|_{\infty} T^{\frac{1-\eta}{2}}+1)$ and some constants $C(\lambda, \eta) , c:=c(\lambda,\eta) >1$. This last inequality is the keystone of the parametrix expansion as it shows that the kernel has a \emph{smoothing effect}. Now, from the continuity of $t \mapsto p^{y}(0,t,x,y)$, by dominated convergence, we may pass to the limit as $\varepsilon$ goes to zero in both sides of \eqref{first:step:param}. We notably obtain
\begin{align}
\forall (x,y) \in \R \times \R^{*}, \ p(0,t,x,y) - \tilde{p}(0,t,x,y) & = \int_0^t  \int_{\R} p(0,s,x,z) (\mathcal{L} - \mathcal{L}^{y}) \tilde{p}(s,t,z,y) dz ds \nonumber \\
 & = p\otimes H(0,t,x,y). \label{first:step:param:limit}
\end{align}

Using \eqref{kernel:bound} and \eqref{first:step:param:limit} together with an induction argument, one obtains the following bound
\begin{align}
\label{iter:step:param}
\forall r \geq0, Ê|\tilde{p} \otimes H^{(r)}(s,t,x,y)| & \leq C^{r+1} (t-s)^{r \eta/2} \prod_{i=1}^{r} B\left( 1 + \frac{(i-1) \eta}{2}, \frac{\eta}{2}\right) g_{c(t-s)}(y-x)
\end{align}

\noindent with $H^{(0)}=I$ and $H^{(r)}=H \otimes H^{(r-1)}$, $r\geq1$ and $B(m,n)=\int_0^1 dv (1-v)^{m-1} v^{n-1} $ is the Beta function. In particular, we may iterate the representation formula \eqref{first:step:param:limit}, in order to obtain
\begin{equation}
\forall (x,y) \in \R \times \R^{*}, \ p(0,t,x,y) = \sum_{r\geq0} \tilde{p} \otimes H^{(r)}(0,t,x,y).
\label{repres:formula:param}
\end{equation}

Since the series converge absolutely and uniformly for $(x,y)\in \R \times \R^*$, we deduce that $x \mapsto p(0,t,x,y)$ is continuous on $\R$ and satisfies the following Gaussian upper bound:
\begin{equation}
\label{Gaussian:upperbound:skewdiff}
\forall (t,x,y) \in (0,T] \times \R \times \R^{*}, Êp(0,t,x,y) \leq E_{\eta/2,1}(C(T^{(\frac{(3-2\eta)}{2}}|b|_{\infty}+T^{1-\frac{\eta}{2}})) g_{c t}(y-x)
\end{equation}

\noindent with $C(\lambda,\eta),\ c(\lambda,\eta)>1$. Once again, for more details, we refer the reader to Section 5 in \cite{kohatsu:taguchi:zhong} and notably to Corollary 5.5 for the extension to the case of bounded measurable drift $b$.

\subsection{Parametrix expansion for the density of the approximation scheme \eqref{Euler:scheme}}
In the spirit of \cite{kona:mamm:00} and \cite{lem:men:2010}, we will take advantage of the discrete counterpart of the parametrix technique to obtain two-sided Gaussian bounds for the density of the discretization scheme \eqref{Euler:scheme}. We first need to introduce a discretization scheme with frozen coefficients and the discrete counterpart of the time-space convolution kernel. This will then allow us to establish the representation for the density of the discretization scheme which is similar to \eqref{repres:formula:param}. In order to do this we first prove that the kernel shares a smoothing property similar to \eqref{kernel:bound}. Finally, as in the previous section, the upper bound will directly follow from the parametrix expansion. 

To derive the global lower-bound, we proceed in two steps. By a scaling argument, without loss of generality, we may assume that $T\leq1$. First, the lower bound is obtained on the diagonal $|y-x|^2/(t_i-t_j) \leq K$, $K>0$, for short time $(t_i-t_j) \leq T < T_0  $, for some $T_0$ to be specified later on. To obtain the off-diagonal bound, we proceed using a chaining argument as usually done in this context, see Chapter VII in Bass \cite{bass:97}, Kusuoka and Stroock \cite{kusu:stro:87} and Lemaire and Menozzi \cite{lem:men:2010} in the case of discretization schemes. We briefly recall these steps for sake of completeness.

We begin our program by introducing the discrete frozen scheme which is the discrete ``analogue'' of \eqref{frozen:proc}. For fixed $x, x' \in \R$, $0\leq j <j' \leq N$, we define the frozen scheme $(\tilde{X}^{N}_{t_i})_{i \in \leftB j,j' \rightB}$ by
\begin{equation}
\label{discretization:frozen:scheme}
\tilde{X}^{N}_{t_j}=x,\ \tilde{X}^{N}_{s} = \tilde{X}^{N}_{t_i} + \sigma(x') (W_{s}-W_{t_i}) + (2\alpha-1) L^{0}_{s-t_i}(\tilde{X}^{N}(t_i,\tilde{X}^{N}_{t_i})),
\end{equation}

\noindent for $s\in [t_i,t_{i+1}]$ and $i=j, \cdots, j'-1$.

From now on, $p_{N}(t_j,t_{j'},x,.)$ and $\tilde{p}_{N}^{t_{j'},x'}(t_j,t_{j'},x,.)$ denote the transition densities between times $t_j$ and $t_{j'}$ of the discretization schemes \eqref{Euler:scheme} and \eqref{discretization:frozen:scheme}. For sake of simplicity, we will use the notation $\tilde{p}_{N}(t_j,t_{j'}, x,x') = \tilde{p}_{N}^{ t_{j'},x'}(t_j,t_{j'},x,x')$.

We also introduce the discrete counterpart of the infinitesimal generators considered so far. For a smooth function $g:\R \rightarrow \R$, a fixed $x'\in \R^{*}$ and $j=0, \cdots, j'-1$, we define
\begin{equation}
\mathcal{L}^{N}_{t_j}g(x) = \frac{\E[g(X^{N}_{t_j +h})|X^{N}_{t_j}=x]-g(x)}{h}, \ \ \mbox{ and } \ \tilde{\mathcal{L}}^{N}_{t_j} g(x) = \frac{\E[g(\tilde{X}^{N}_{t_j +h}) | \tilde{X}^{N}_{t_j}=x]-g(x)}{h}
\label{discrete:generator}
\end{equation}

\noindent for $0\leq j <j' \leq N$, the discrete kernel
\begin{equation}
\label{discrete:generator}
H_{N}(t_j,t_{j'},x,x') = \left(\left(\mathcal{L}^{N}_{t_j} - \tilde{\mathcal{L}}^{N}_{t_j}\right) \tilde{p}(t_j+h,t_{j'},.,x')\right)(x)
\end{equation}

\noindent and finally the discrete time-space convolution type operator $\otimes_{N}$ as follows
$$
(g\otimes_N f)(t_j,t_{j'},x,x') = \sum_{i=j}^{j'-1} h \int_{\R} g(t_j,t_{i},x,z) f(t_{i},t_{j'},z,x') dz
$$

\noindent with the convention that $\sum_{i=j}^{j'-1}\cdots=0$ if $j\geq j'$. From \eqref{discrete:generator}, it is easily seen that
$$
H_{N}(t_j,t_{j'},x,x') = h^{-1} \int_\R (p_{N} - \tilde{p}_{N}^{t_{j'},x'})(t_{j},t_{j+1},x,z) \tilde{p}(t_{j+1},t_{j'},z,x') dz
$$

Analogously to \cite{kona:mamm:00}, we define the convolution as follows: $g\otimes_N H^{(0)}_{N} = g$ and for $r\geq1$, $g\otimes_{N} H^{(r)}_N = (g\otimes_N H^{(r-1)}_N) \otimes_N H_N$. The following parametrix expansion of $p_N$ follows from the same arguments to those employed in Lemma 3.6, \cite{kona:mamm:00}. The proof is omitted.
\begin{PROP}\label{prop:param:exp}For $0\leq j < j' \leq N$, one has
\begin{equation}
\label{param:series:euler:scheme}
\forall (x,x')\in \R\times \R^{*}, \ p_N(t_j,t_{j'},x,x') = \sum_{r=0}^{j'-j} \left(\tilde{p} \otimes_N H^{(r)}_N\right)(t_j,t_{j'},x,x') 
\end{equation}

\noindent where we use the convention $\tilde{p}(t_{j'},t_{j'},.,x') = \tilde{p}(t_{j'},t_{j'},.,x')=\delta_{x'}(.)$ in the computation of $\tilde{p} \otimes_N H^{(r)}_N$.
\end{PROP}

Next result is the keystone for the proof of Theorem \ref{gauss:bound:scheme} as it provides the smoothing property of the discrete convolution kernel.
\begin{lmm}\label{param:iterate:density:scheme:lemma}Under \A{HR} and \A{HE}, there exist a constant $c:=c(\lambda,\eta)>1$ such that for all $0\leq j < j' \leq N$, for all $r \in \leftB0, j'-j \rightB$, one has
\begin{equation}
\label{param:discrete:iterate}
\forall (x,x') \in \R \times \R^{*}, \ |\tilde{p} \otimes_N H^{(r)}_N(t_j,t_{j'},x,x')| \leq  C^{r} (t_{j'}-t_{j})^{r \eta/2} \prod_{i=1}^{r} B\left(1+\frac{(i-1)\eta}{2},\frac{\eta}{2}\right) g_{c(t_{j'}-t_j)}(x'-x)
\end{equation}

\noindent where $T \mapsto  C:=C(T,b,\sigma)$ is a non-decreasing positive function.
\end{lmm}

\begin{proof} We first prove that for all $0\leq j < j' \leq N$, one has
\begin{equation}
\label{key:estimate:discrete:param}
\forall (x,x') \in \R \times \R^{*}, \ |H_{N}(t_j,t_{j'},x,x') | \leq C (t_{j'}-t_{j})^{-1 +\frac{\eta}{2}} g_{c (t_{j'} - t_j)}(x'-x)
\end{equation}

\noindent for a positive constant $c(\lambda,\eta)>1$ and non-decreasing positive function $T\mapsto C(T,b,\sigma)$. For $j'=j+1$, by time-homogeneity, we remark that
$$
\forall (x,x') \in \R \times \R^{*}, \ H_{N}(t_j,t_{j'},x,x') = h^{-1} (p_{N} - \tilde{p})(t_{j},t_{j+1},x,x') = h^{-1} (p_{N} - \tilde{p})(0,h,x,x').
$$

\noindent Let $s\in (0,h]$ fixed. We will prove a more general result, namely
\begin{equation}
\label{diff:p:tildep}
\forall (x,x',z) \in \R \times \R \times \R^{*}, \ \ |(p_{N} - p^{x'})(0,s,x,z)|\leq C (|b|_{\infty} s^{\frac12} +  |x-x'|^{\eta}) g_{cs}(z-x).
\end{equation}

 Then, the result will follow from \eqref{diff:p:tildep} by setting $s=h$, $z=x'$ and by using the space-time inequality: $\forall \lambda>0, \forall p>0, \exists C>0$  s.t. $ \forall x>0,\ x^{p}\exp(-\lambda x)\leq C$. Indeed, from \eqref{diff:p:tildep} with $s=h$, $z=x'$ and the space-time inequality, one gets $|(p_{N} - p^{x'})(0,h,x,x')| = | (p_{N} - \tilde{p})(0,h,x,x') | \leq C (|b|_{\infty} h^{\frac12} + |x-x'|^{\eta}) g_{c h}(x'-x) \leq C (|b|_{\infty} h^{\frac12} + h^{\frac{\eta}{2}}) g_{c h}(x'-x)$ so that \eqref{key:estimate:discrete:param} follows for $j'=j+1$. We now prove \eqref{diff:p:tildep}. Using the exact expression of the density of the skew Brownian motion, we separate the computations in the following eight cases:

\noindent $\bullet$ For $z \geq0$ and $x, x+b(x)s\geq 0$, one has
\begin{align*}
(p_{N} - \tilde{p}^{x'})(0,s,x,z) & = \left\{ g_{a(x)s}(z-x-b(x)s)-g_{a(x')s}(z-x) + (2\alpha-1) \left( g_{a(x) s}(z+x+b(x)s) - g_{a(x')s}(z+x)\right) \right\}.
\end{align*} 
\noindent $\bullet$ For $z \geq0$ and $x\geq 0$, $x+b(x)s < 0$, one has
\begin{align*}
(p_{N} - \tilde{p}^{x'})(0,s,x,z) & =  \left\{ g_{a(x)s}(z-x-b(x)s)-g_{a(x')s}(z-x) + (2\alpha-1) \left( g_{a(x) s}(z-x-b(x)s) - g_{a(x')s}(z+x)\right) \right\}.
\end{align*} 

\noindent $\bullet$ For $z \geq0$ and $x< 0$, $x+b(x)s< 0$, one has
\begin{align*}
(p_{N} - \tilde{p}^{x'})(0,s,x,z)   & =  2 \alpha \left\{ g_{a(x)s}(z-x-b(x)s)-g_{a(x')s}(z-x)\right\}.
\end{align*} 

\noindent $\bullet$ For $z \geq0$ and $x< 0$, $x+b(x)s\geq 0$, one has
\begin{align*}
(p_{N} - \tilde{p}^{x'})(0,s,x,z)  & =  \left\{ g_{a(x)s}(z-x-b(x)s) - g_{a(x')s}(z-x) + (2\alpha-1) \left( g_{a(x) s}(z-x-b(x)s) - g_{a(x')s}(z-x)\right) \right\}.
\end{align*} 

\noindent $\bullet$ For $z <0$ and $x, x+b(x)s\geq 0$, one has
\begin{align*}
(p_{N} - \tilde{p}^{x'})(0,s,x,z)   & = 2 (1-\alpha) \left\{ g_{a(x)s}(z-x-b(x)s) - g_{a(x')s}(z-x)  \right\}.
\end{align*} 

\noindent $\bullet$ For $z <0$ and $x\geq 0$, $x+b(x)s < 0$, one has
\begin{align*}
(p_{N} - \tilde{p}^{x'})(0,s,x,z)   & =  \left\{ g_{a(x)s}(z-x-b(x)s)-g_{a(x')s}(z-x) + (1-2\alpha) \left( g_{a(x) s}(z + x + b(x)s) - g_{a(x')s}(z-x)\right) \right\} .
\end{align*}

\noindent $\bullet$ For $z <0$ and $x< 0$, $x+b(x)s < 0$, one has
\begin{align*}
(p_{N} - \tilde{p}^{x'})(0,s,x,z)   & =  \left\{ g_{a(x)s}(z-x-b(x)s)-g_{a(x')s}(z-x) + (1-2\alpha) \left( g_{a(x) s}(z + x + b(x)s) - g_{a(x')s}(z+x)\right) \right\}.
\end{align*}

\noindent $\bullet$ For $z <0$ and $x< 0$, $x+b(x)s \geq 0$, one has
\begin{align*}
(p_{N} - \tilde{p}^{x'})(0,s,x,z)   & =  \left\{ g_{a(x)s}(z-x-b(x)s)-g_{a(x')s}(z-x) + (1-2\alpha) \left( g_{a(x) s}(z - x - b(x)s) - g_{a(x')s}(z + x)\right) \right\} .
\end{align*} 

We treat the case $z\geq0$, $x \geq 0$, $x+b(x)s < 0$. The estimates for the other cases are similarly obtained and details are omitted. First, from \A{HR} and \A{HE}, one easily gets
$$
| g_{a(x)s}(z-x-b(x)s)-g_{a(x')s}(z-x)| \leq C(|b|_{\infty} s^{\frac12} + |x-x'|^{\eta}) g_{c s}(x'-x)
$$

\noindent for some positive constant $c(\lambda,\eta)>1$ and a non decreasing positive function $T\mapsto C:=C(T,b,\sigma)$.

Moreover, observe that $0 \leq x \leq |b|_{\infty}s$, the signs of $x$ and $z$ being the same, $g_{c s}(z+x) \leq g_{c s}(z-x)$, so that using similar arguments, one gets
\begin{align*}
|g_{a(x) s}(z-x-b(x)s) - g_{a(x')s}(z+x) | & \leq |g_{a(x) s}(z+x - (2x +b(x)s)) - g_{a(x') s}(z+x - (2x +b(x)s))|  \\
& +  |g_{a(x') s}(z+x - (2x +b(x)s)) - g_{a(x')s}(z+x) | \\
& \leq C (|b|_\infty s^{\frac12} + |x-x'|^{\eta}) g_{cs}(z+x) \leq C (|b|_\infty s^{\frac12} + |x-x'|^{\eta}) g_{c s }(z-x). 
\end{align*}

From the above computations, one gets \eqref{diff:p:tildep} and estimate \eqref{key:estimate:discrete:param} clearly follows. Note also that from the previous computations, it also follows that 
\begin{equation}
\forall (s,x,z)\in (0,h] \times \R \times \R, \ p_N(0,s,x,z) \leq C g_{c s}(z-x),
\label{bound:transition:pN}
\end{equation}

\noindent with $C(T,b, \sigma), c(\lambda,\eta)>1$ independent of $N$.

For $j'>j+1$, we remark that $z \mapsto \tilde{p}(t_{j+1},t_{j'},z,x') \in \mathcal{D}(\alpha)$ so that from the symmetric It-Tanaka formula, for all $(x,x')\in \R \times\R^{*}$, one has
\begin{align}
H_{N}(t_j, t_{j'},x,x')  & = h^{-1}\left\{ \E[\tilde{p}(t_{j+1},t_{j'},X^{N,t_{j},x}_{t_{j+1}},x')] - \E[\tilde{p}(t_{j+1},t_{j'},\tilde{X}^{N,t_{j},x}_{t_{j+1}},x')]\right\} \nonumber \\
& = h^{-1}  \int_0^h \left[ b(x)\E[\partial_x\tilde{p}(t_{j+1},t_{j'},X^{N,t_{j},x}_{s},x')] + \frac12 a(x) \E[\partial^2_x\tilde{p}(t_{j+1},t_{j'},X^{N,t_{j},x}_{s},x')] \right] ds\nonumber \\
&  \quad - h^{-1}  \int_0^h \frac12 a(x') \E[\partial^2_x\tilde{p}(t_{j+1},t_{j'}, \tilde{X}^{N,t_{j},x}_{s},x')]  ds \label{HN:without:remainder} \\
& = b(x) \partial_x \tilde{p}(t_{j+1},t_{j'},x,x') + \frac12 (a(x)-a(x'))\partial^{2}_x \tilde{p}(t_{j+1},t_{j'},x,x') + h^{-1} \mathscr{R}_{N}(t_j,t_{j'},x,x') \label{HN}
\end{align}

\noindent with 
\begin{align}
\mathscr{R}_{N}(t_j,t_{j'},x,x') & := b(x) \int_0^h \left\{ \E[\partial_x\tilde{p}(t_{j+1},t_{j'},X^{N,t_{j},x}_{s},x')] - \partial_x\tilde{p}(t_{j+1},t_{j'},x,x') \right\} ds \nonumber \\
 & \quad + \frac12 a(x) \int_0^h \left\{ \E[\partial^2_x\tilde{p}(t_{j+1},t_{j'},X^{N,t_{j},x}_{s},x')] - \partial^2_x\tilde{p}(t_{j+1},t_{j'},x,x') \right\} ds \label{RN:ito:tanaka} \\
 & \quad - \frac12 a(x') \int_0^h \left\{ \E[\partial^2_x\tilde{p}(t_{j+1},t_{j'},\tilde{X}^{N,t_{j},x}_{s},x')] - \partial^2_x\tilde{p}(t_{j+1},t_{j'},x,x') \right\} ds \nonumber .
\end{align}

In order to prove \eqref{key:estimate:discrete:param}, we look at the first term appearing in the right-hand side of \eqref{HN:without:remainder}. Distinguishing the four cases $x,x' \geq 0$; $x\geq 0,x'<0$; $x<0,x' \geq 0$ and $x<0,x'<0$ and using \A{HE}, \A{HR} simple but cumbersome computations show that for $s\in (0,h]$
\begin{align*}
| \E[\partial_x\tilde{p}(t_{j+1},t_{j'},X^{N,t_{j},x}_{s},x')]  | & \leq \int_{\R}p_N(t_j,t_j+s,x,z) |\partial_z\tilde{p}(t_{j+1},t_{j'},z,x')| dz \\
& \leq C (t_{j'}-t_{j+1})^{-\frac12} g_{c (t_{j'}-t_{j+1}+s)}(x'-x) \\
& \leq C (t_{j'}-t_{j})^{-\frac12} g_{c (t_{j'}-t_{j})}(x'-x) 
\end{align*}

\noindent where we used \eqref{bound:transition:pN} and the inequality $t_{j'}-t_{j}\geq 2 h$ for the last inequality. Hence it follows
$$
| h^{-1}  \int_0^h  b(x)\E[\partial_x\tilde{p}(t_{j+1},t_{j'},X^{N,t_{j},x}_{s},x')]  ds | \leq C|b|_{\infty} (t_{j'}-t_{j})^{-\frac12} g_{c (t_{j'}-t_{j})}(x'-x).
$$

For the two remaining terms appearing in the right-hand side of \eqref{HN:without:remainder}, we use the decomposition 
\begin{align*}
 h^{-1} \int_0^h & \left( \frac12 a(x) \E[\partial^2_x\tilde{p}(t_{j+1},t_{j'},X^{N,t_{j},x}_{s},x')]  - \frac12 a(x') \E[\partial^2_x\tilde{p}(t_{j+1},t_{j'}, \tilde{X}^{N,t_{j},x}_{s},x')] \right) ds  \\
  & =  h^{-1} \int_0^h \frac12 (a(x)-a(x')) \E[\partial^2_x\tilde{p}(t_{j+1},t_{j'},X^{N,t_{j},x}_{s},x')]  ds \\
 & \quad + h^{-1} \int_0^h \frac12 a(x') (\E[\partial^2_x\tilde{p}(t_{j+1},t_{j'},X^{N,t_{j},x}_{s},x')]- \E[\partial^2_x\tilde{p}(t_{j+1},t_{j'}, \tilde{X}^{N,t_{j},x}_{s},x')])  ds
\end{align*}

\noindent and observe that using computations similar to the first term and \A{HR}, one gets
$$
| (a(x)-a(x')) \E[\partial^2_x\tilde{p}(t_{j+1},t_{j'},X^{N,t_{j},x}_{s},x')] | \leq C \frac{|x-x'|^\eta}{(t_{j'}-t_j)} g_{c (t_{j'}-t_{j})}(x'-x) \leq C (t_{j'}-t_j)^{-1+\frac{\eta}{2}} g_{c (t_{j'}-t_{j})}(x'-x).
$$

For the last term, from \eqref{diff:p:tildep} and the semigroup property, we obtain
\begin{align*}
| \E[\partial^2_x\tilde{p}(t_{j+1},t_{j'},X^{N,t_{j},x}_{s},x')]- \E[\partial^2_x\tilde{p}(t_{j+1},t_{j'}, \tilde{X}^{N,t_{j},x}_{s},x')] | & \leq \int_{\R} |p_N(0,s,x,z)-\tilde{p}^{x'}(0,s,x,z)| | \partial^2_z \tilde{p}(t_{j+1},t_{j'},z,x') | dz \\
& \leq C (|b|_{\infty} s^{\frac12}+|x-x'|^{\eta}) (t_{j'}-t_{j+1})^{-1} g_{c(t_{j'}-t_{j+1}+s)}(x'-x) \\
& \leq C (t_{j'}-t_{j})^{-1+\frac{\eta}{2}} g_{c(t_{j'}-t_j)}(x'-x)
\end{align*}

\noindent where we again used $t_{j'}-t_{j}\geq 2 h$ and $T\mapsto C(T,b,\sigma)$ is a non-decreasing positive function. Consequently, from the previous computations, we derive the following bound
\begin{align*}
|H_{N}(t_j, t_{j'},x,x')|  & \leq C(T,b,\sigma) (t_{j'}-t_j)^{-1+ \frac{\eta}{2}} g_{c(t_{j'}-t_{j})}(x'-x)
\end{align*}

\noindent which in turn allows to conclude on the validity of \eqref{key:estimate:discrete:param} for all $0\leq j < j'\leq N$. Now one clearly gets
\begin{align*}
|\left(\tilde{p} \otimes_N H_N\right)(t_j,t_{j'},x,x') | & \leq \sum_{i=j}^{j'-1} h \int_{\R} |\tilde{p}(t_j,t_i, x,z)| |H_{N}(t_i, t_{j'},z,x')| dz \\
& \leq C(T,b,\sigma)  \left\{\sum_{i=j}^{j'-1} h (t_{j'}-t_{i})^{-1+\frac{\eta}{2}} \right\} g_{c(t_{j'}-t_{j})}(x'-x) \\
& \leq C(T,b,\sigma) (t_{j'}-t_{j})^{\eta/2} B(1, \frac{\eta}{2}) g_{c(t_{j'}-t_{j})}(x'-x) 
\end{align*}

\noindent where we used the Gaussian upper-bound $|p^{z}(t_{j},t_i,x,z) | \leq C g_{c(t_i-t_{j})}(z-x)$ and the semigroup property for the last but one inequality. Finally, estimate \eqref{param:discrete:iterate} follows by induction.
\end{proof}

\subsection{Proof of Theorem \ref{gauss:bound:scheme}}
The Gaussian upper bound in Theorem \ref{gauss:bound:scheme} directly follows from Proposition \ref{prop:param:exp} and the following bound 
\begin{align*}
\sum_{r\geq0} C^{r} (t_{j'}-t_{j})^{r \eta/2} \prod_{i=1}^{r+1} B\left(1+\frac{(i-1)\eta}{2},\frac{\eta}{2}\right) & \leq \sum_{r\geq0} \left( C \Gamma(\eta/2)(t_{j'}-t_{j})^{\eta/2}\right)^{r} \frac{1}{\Gamma(1+r\eta/2)} \\
& \leq  E_{\eta/2,1}( C(T,b,\sigma)T^{\eta/2}) < \infty.
\end{align*}

It now remains to prove the Gaussian lower-bound. As already mentioned in the beginning of Section \ref{Euler:scheme}, w.l.o.g. we assume that $T\leq 1$. The proof of the lower bound does not depend on the structure of the equation but only relies on the parametrix expansion obtained in Lemma \ref{prop:param:exp} with the underlying Gaussian controls \eqref{param:discrete:iterate}. Hence, we closely follow the arguments in \cite{lem:men:2010} and omit some part of the proof. It is decomposed into two steps that we explain here for sake of completeness. From Proposition \ref{prop:param:exp}, we first obtain the lower-bound in short time $T$ in the diagonal regime that is on compact sets of the underlying Gaussian time-space metric. The Òoff-diagonalÓ bound is then obtained by a chaining argument, see \cite{bass:97}, \cite{kusu:stro:87} and \cite{lem:men:2010}. Finally, to derive the lower bound for an arbitrary fixed time $T$, one may use the semigroup property.

From Proposition \ref{prop:param:exp}, one has
$$
p_{N}(t_j,t_{j'},x,x') = \tilde{p}_{N}(t_{j},t_{j'},x,x') + \bar{p}_{N}(t_{j},t_{j'},x,x'), 
$$

\noindent with $\bar{p}_{N}(t_{j},t_{j'},x,x') := \sum_{r=1}^{j'-j} \left(\tilde{p} \otimes_N H^{(r)}_N\right)(t_j,t_{j'},x,x') $ which satisfies
$$
|\bar{p}_{N}(t_j,t_{j'},x,x')| \leq C (t_{j'}-t_{j})^{\eta/2} g_{c(t_{j'}-t_{j})}(x'-x)
$$

\noindent for some positive constants $C:=C(T,b,\sigma, \lambda),\ c(\lambda,\eta)>1$.

\noindent $\bullet$ \emph{Step 1: Diagonal regime in short time}

From the above decomposition, one has
\begin{align}
p_{N}(t_j,t_{j'},x,x') & \geq C^{-1} g_{c^{-1}(t_{j'}-t_{j})}(x'-x) - C (t_{j'}-t_{j})^{\eta/2} g_{c(t_{j'}-t_{j})}(x'-x) \label{first:step:lower:bound:diagregime}
\end{align}

 \noindent where we used $\tilde{p}_{N}(t_{j},t_{j'},x,x') \geq C^{-1} g_{c^{-1}(t_{j'}-t_{j})}(x'-x)$ by \A{HE}. Let $R\geq1/2$. For $x,x'$ s.t. $|x'-x|^2/(t_{j'}-t_j)\leq 2 R$ and $t_{j'}-t_j \leq T \leq (c\exp(-cR)/(2 C^2 ))^{2/\eta}$, we get
 \begin{align}
 p_{N}(t_j,t_{j'},x,x') & \geq  \frac{C^{-1}}{\sqrt{2\pi (t_{j'}-t_j)}} \left( c^{1/2} \exp\left(-c R \right) - \frac{C^2}{c^{1/2}}(t_{j'}-t_{j})^{\eta/2}\right) \nonumber \\
 & \geq \frac{C^{-1} c^{1/2}}{2 \sqrt{2\pi (t_{j'}-t_j)}}  \exp\left(-c R \right) \nonumber \\
 & \geq  \frac{C^{-1}}{\sqrt{2\pi (t_{j'}-t_j)}} \label{short:time:diag:regime:decay} \\
 & \geq C^{-1} g_{c^{-1}(t_{j'}-t_{j})}(x'-x) \nonumber
 \end{align}
 
 \noindent for some constant $c>1$ and up to a modification of the constant $C$. Hence, there exist two constants $c,C>1$ s.t. for $|x'-x|^2/(t_{j'}-t_j)\leq 2 R$ one has $p_{N}(t_j,t_{j'},x,x') \geq C^{-1} g_{c^{-1}(t_{j'}-t_{j})}(x'-x)$. 
 
 We now extend this first result to arbitrary times $0 \leq s \leq t \leq T$ not necessarily corresponding to discretization times of the scheme. For $0 \leq s < t \leq T$ that do not belong to the time grid and $(x,y,x') \in \R^2 \times \R^{*}$, we define the kernel $p_{N}^{y}(s,t,x,x')$ s.t. for all bounded measurable $f$, $\E[f(X^{N}_t) | X^{N}_{\phi_N(s)} =y, X^{N}_{s}=x] = \int_{\R} f(x') p_{N}^{y}(s,t,x,x') dx'$. More precisely, we prove
\begin{equation}
\label{lower:bound:diag:continuous:time}
 \exists C>1, \forall 0 \leq s < t \leq T, \forall y \in \R, \frac{|x'-x|^2}{(t-s)}\leq \frac{R}{12}, \  p^{N,y}(s,t,x,x') \geqÊ \frac{C^{-1}}{\sqrt{2\pi (t-s)}}.
\end{equation}
 
 If $s$ or $t$ belong to the time grid of the scheme, the following lines of reasoning may be easily adapted. From the semigroup property, assuming that $\phi_N(t) - (\phi_N(s) + h) \geq h$, one has
\begin{equation}\label{lower:bound:before:chaining}
 p_N^{y}(s,t,x,x') \geq \int_{B_{R_0}(s,t,x,x')} p_N^{y}(s,\phi_N(s)+h,x,x_1) p_{N}(\phi_N(s)+h,\phi_N(t), x_1,x_2) p_{N}(\phi_N(t),t,x_2,x') dx_1 dx_2
\end{equation}
 
 \noindent where $B_{R_0}(s,t,x,x') := \left\{x_1 \in \R^{*} : \frac{|x_1-x|^2}{(\phi_{N}(s)+h-s)} \leq R_0 \right\} \times \left\{x_2 \in \R : \frac{|x'-x_2|^2}{(t-\phi_N(t))} \leq R_0 \right\}$ for some $R_0>0$. Now, for $(x_1,x_2)\in \B_{R_0}(s,t,x,x')$, we remark that 
$$
\frac{|x_2-x_1|^2}{(\phi_N(t)-(\phi_{N}(s)+h))} \leq \frac{2 |x_2-x'|^2 + 4 |x'-x |^2 + 4 |x-x_1|^2}{(\phi_N(t)-(\phi_{N}(s)+h))} \leq 6 R_0 + R
$$ 
 
 \noindent where we used $\phi_N(t)-(\phi_{N}(s)+h)) \geq h \geq \phi_N(s)+h -s, t-\phi_N(t)$ and $1/(\phi_N(t)-(\phi_{N}(s)+h)) \leq 3/(t-s)$. Now we set $R_0 = R/2$ so that we have $|x_2-x_1|^2/(\phi_N(t)-(\phi_{N}(s)+h)) \leq 2 R$. Consequently, combining \eqref{short:time:diag:regime:decay} with \eqref{lower:bound:before:chaining},  one gets
$$
p_{N}^{y}(s,t,x,x') \geq C^{-1} (\phi_N(s)+h - s)^{-1/2} (\phi_N(t)-\phi_N(s)-h)^{-1/2} (t-\phi_N(t))^{-1/2} |B_1| |B_2|
$$

\noindent with $B_1 = \left\{x_1 \in \R^{*} : \frac{|x_1-x|^2}{(\phi_{N}(s)+h-s)} \leq R_0 \right\}$, $B_2 = \left\{x_2 \in \R : \frac{|x'-x_2|^2}{(t-\phi_N(t))} \leq R_0 \right\}$ and $|.|$ stands for the Lebesgue measure on $\R$. Finally, since $|B_1| \geq \rho \sqrt{(\phi_{N}(s)+h-s)}$, $|B_2| \geq \rho \sqrt{(t-\phi_N(t))}$ for some $\rho>0$ and $\phi_N(t)-\phi_N(s)-h \leq (t-s)$, we deduce \eqref{lower:bound:diag:continuous:time} if $\phi_N(t) - (\phi_N(s) + h) \geq h$. The case $\phi_N(t) - (\phi_N(s) + h) < h$ is more easily obtained and details are omitted.

\noindent $\bullet$ \emph{Step 2: Chaining and off-diagonal regime in short time}

It remains to deal with the off-diagonal regime that is the case when $0\leq j < j' \leq N$, $(x,x') \in \R\times\R^{*}$, $\frac{|x'-x|^2}{t_{j'}-t_j}\geq 2 R \geq1$. We set $M := \lceil K \frac{|x'-x|^2}{t_{j'}-t_j}  \rceil$, for some $K\geq1$ to be specified later on, $\delta : = (t_{j'}-t_j)/M$ and introduce a new space-time grid $y_k = x + \frac{k}{M}(x'-x)$, $s_{k}=t_j + k \delta$, $k\in \leftB0,M\rightB$. Note that $y_0=x$, $y_M = x'$, $s_0=t_j$ and $s_M = t_{j'}$. We also define for $k\in \leftB1,M-1\rightB$, $B_k =\left\{ x \in \R : |x-y_k| \leq \rho\right\}$, $\rho = |x'-x|/M$ so that $|y_k-y_{k-1}| = |x'-x|/M$, $k \in \leftB1, M \rightB$. Consequently, we deduce
\begin{align}
\forall x \in B_1, |x-x_1| \leq 2 \rho, & \forall (x_k,x_{k+1}) \in B_k \times B_{k+1}, |x_k - x_{k+1}| \leq 3 \rho,\  k=2, \cdots, M-2 \nonumber \\
& \forall x_M\in B_{M}, |x'-x_M| \leq 2\rho. \label{path:deterministic:curve}
\end{align}

We now choose $K$ large enough s.t.
\begin{equation}
\label{set:number:point:grid}
\frac{3\rho}{\sqrt{\delta}}  =  \frac{3 |x'-x|}{\sqrt{M}(t_{j'}-t_j)} \leq \sqrt{\frac{R}{12}}
\end{equation}

\noindent which according to \eqref{lower:bound:diag:continuous:time} yields for $k=2, \cdots, M-2 $
\begin{align}
\forall x_1 \in B_1, p_{N}(s_0,s_1,x,x_1) \geq C^{-1} \delta^{-1/2}, & \forall (x_k,x_{k+1},y) \in B_k \times B_{k+1} \times \R, p_{N}^{y}(s_k,s_{k+1},x_k,x_{k+1}) \geq C^{-1}\delta^{-1/2},\nonumber \\
& \forall (x_{M-1},y) \in B_{M} \times \R, p_{N}^{y}(s_{M-1},s_{M}, x_{M-1},x') \geq C^{-1}\delta^{-1/2}. \nonumber
\end{align}

Before going further we have to distinguish two cases: $\delta<h$ and $\delta\geq h$. We will treat the first one. The case $\delta\geq h$ can be treated with similar arguments. We refer to \cite{bass:97} or \cite{lem:men:2010} for more details. We introduce $I_k :=\left\{\ell \in \leftB0,M-1\rightB : s_{\ell} \in [t_k, t_{k+1})\right\}$, $d_k:= \sharp I_k$, $i \in \leftB1,d_k\rightB$,  $I^{i}_k \in I_k$ with $t_{k} \leq s_{I^{1}_k} \leq \cdots \leq s_{I^{d_k}_k} < t_{k+1}$. Now we write
\begin{align}
p_{N}(t_j,t_{j'},x,x') & \geq 
  \E\left[\mbox{\textbf{1}}_{\left\{ \cap_{k=j}^{j'-2}\cap_{i\in I_k}  X^{N}_{s_i} \in B_i\right\}} P_{j'-1,j} | X^{N}_{t_j}=x \right] \label{step:recur:lower:bound:off:diag}
\end{align}

\noindent with 
\begin{align*}
P_{j'-1,j} & := \E[\mbox{\textbf{1}}_{\left\{\cap_{i\in I_{j'-1}}X^{N}_{s_{i}}\in B_i\right\}}  p_{N}^{X^{N}_{\phi^{N}(s_{M-1})}}(s_{M-1},t_{j'},X^{N}_{s_{M-1}},x')| \mathcal{F}_{s_{I^{d_{j'-2}}_{j'-2}}} ] \\
& = \E[\mbox{\textbf{1}}_{\left\{X^{N}_{s_{I^{1}_{j'-1}}}\in B_{I^{1}_{j'-1}}\right\}} \int_{\prod^{d_{j'-1}}_{i=2}B_{I^{i}_{j'-1}}} p_{N}^{X^{N}_{t_{j'-1}}}(s_{I^{1}_{j'-1}},s_{I^{2}_{j'-1}},X^{N}_{s_{I^{1}_{j'-1}}},x_2)  \\
& \times \prod_{i=2}^{d_{j'-1}}  p_{N}^{X^{N}_{t_{j'-1}}}(s_{I^{i}_{j'-1}},s_{I^{i+1}_{j'-1}},x_i,x_{i+1}) p_{N}^{X^{N}_{t_{j'-1}}}(s_{I^{d_{j'-1}}_{j'-1}},t_{j'},x_{d_{j'-1}},x') \prod_{i=2}^{d_{j'-1}}  dx_i | \mathcal{F}_{s_{I^{d_{j'-2}}_{j'-2}}} ] 
\end{align*}

\noindent so that, noting that $\exists c>0$ s.t. $\forall i \in \leftB 1,M-1 \rightB, |B_i|\geq c \rho$ from \eqref{path:deterministic:curve}, \eqref{set:number:point:grid} and \eqref{lower:bound:diag:continuous:time}, on $\left\{X^{N}_{s_{I^{d_{j'-2}}_{j'-2}}} \in B_{I^{d_{j'-2}}_{j'-2}}\right\}$, one gets
\begin{align*}
P_{j'-1,j} & \geq (C^{-1} \delta^{-1/2})^{d_{j'-1}} (c \rho)^{d_{j'-1}-1} \int_{B_{I^{1}_{j'-1}}} p^{N,X^{N}_{t_{j'-2}}}(s_{I^{d_{j'-2}}_{j'-2}}, s_{I^{1}_{j'-1}}, X^{N}_{s_{I^{d_{j'-2}}_{j'-2}}},x_1) dx_1 \\
& \geq (C^{-1} \delta^{-1/2})^{d_{j'-1}+1} (c \rho)^{d_{j'-1}}.
\end{align*}

Plugging this estimate in \eqref{step:recur:lower:bound:off:diag}, we finally obtain
$$
p_{N}(t_j,t_{j'},x,x') \geq (C^{-1} \delta^{-1/2})^{d_{j'-1}+1} (c \rho)^{d_{j'-1}}\E\left[\mbox{\textbf{1}}_{\left\{ \cap_{k=j}^{j'-2}\cap_{i\in I_k}  X^{N}_{s_i} \in B_i\right\}}  | X^{N}_{t_j}=x \right]
$$

\noindent so that, by induction, up to a modification of $C,c>0$, we get
\begin{align*}
p_{N}(t_j,t_{j'},x,x') & \geq  (C^{-1} \delta^{-1/2})^{M+1} (c \rho)^{M} \\
& \geq C^{-1} (t_{j'}-t_j)^{-1/2} \exp\left(M \log(C^{-1}c(\rho/\sqrt{\delta}))\right) \\
& \geq C^{-1} (t_{j'}-t_j)^{-1/2} \exp\left(- c \frac{|x'-x|^2}{(t_{j'}-t_j)}\right) := C^{-1} g_{c^{-1}(t_{j'}-t_j)}(x'-x)
\end{align*}

\noindent for $\frac{|x'-x|^2}{t_{j'}-t_j}\geq 2 R$. This last bound concludes the proof of Theorem \ref{gauss:bound:scheme}.

\subsection{Proof of Theorem \ref{weak:error:scheme}}

 \noindent $\bullet$ \emph{Strategy of proof:} \
\vspace*{.3cm}

In order to prove an error bound for the difference between the transition densities of \eqref{skew:diff} and \eqref{Euler:scheme}, our strategy is the following. The main point is to compare the two series \eqref{repres:formula:param} and \eqref{param:series:euler:scheme} which differs on account of the discrete nature of time-space convolution operator $\otimes_N$ and the discrete smoothing kernel $H_N$. In order to do this, we introduce for $0\leq j < j' \leq N$,
\begin{equation}
\label{density:param:discrete:convol}
\forall (x,x') \in \R\times \R^*, \,Êp^{d}_N (t_j,t_{j'},x,x') = \sum_{r\geq0} \tilde{p} \otimes_N H^{(r)}(t_j,t_{j'},x,x').
\end{equation}

Arguments similar to those of Lemma \ref{param:iterate:density:scheme:lemma} show that the series \eqref{density:param:discrete:convol} converge absolutely and uniformly on $\R \times \R^*$ and that $p^{d}_N$ satisfies the following Gaussian upper-bound:
\begin{equation}
\label{gaussian:upper:bound:pnd}
\forall (x,x') \in \R\times \R^*, \, p^{d}_N (t_j,t_{j'},x,x') \leq E_{\eta/2,1}( C(|b|_{\infty}T^{\frac12}+T^{\frac{\eta}{2}}) ) g_{c (t_{j'}-t_j)}(x'-x).
\end{equation}

Indeed, from \eqref{kernel:bound} and the semigroup property, one gets $|\tilde{p} \otimes_N H(t_j,t_{j'},x,x')| \leq C(|b|_{\infty}T^{\frac{1-\eta}{2}}+1)(t_{j'}-t_{j})^{\eta/2} B(1, \frac{\eta}{2}) g_{c(t_{j'}-t_{j})}(x'-x)$ which in turn by induction yields for all $r\geq1$
\begin{equation}
\label{param:discrete:iterate:H}
\forall (x,x') \in \R \times \R^{*}, \ |\tilde{p} \otimes_N H^{(r)}(t_j,t_{j'},x,x')| \leq  C^{r} (t_{j'}-t_{j})^{r \eta/2} \prod_{i=1}^{r} B\left(1+\frac{(i-1)\eta}{2},\frac{\eta}{2}\right) g_{c(t_{j'}-t_j)}(x'-x)
\end{equation} 

with $C:= C(\lambda,\eta)(|b|_{\infty}T^{\frac{1-\eta}{2}}+1)$ and \eqref{gaussian:upper:bound:pnd} follows. We omit technical details.

The idea is now to decompose the global error as follows: 
$$
(p-p_N)(t_j,t_i,x,y) = (p-p^{d}_N)(t_j,t_i,x,y) + (p^{d}_N-p_N)(t_j,t_i,x,y).
$$
 
 A similar decomposition has been used in \cite{kona:mamm:02} when the coefficients $b$ and $\sigma$ are smooth. The smoothness of the coefficients notably allows to use Taylor expansions to express the transition density $p$ as the fundamental solution of the underlying parabolic PDE and to use integration by parts (that may be expressed as the duality relation satisfied by the infinitesimal generator when seen as a differential operator) in order to equilibrate time singularities. Obviously, these arguments do not work here under the mild smoothness assumption \A{HR} so that computations become more delicate. In a first step, we express the difference $p^d_N-p_N$ in an infinite parametrix series that involves the difference between the two kernels $H$ and $H_N$. Then, the symmetric It\^o-Tanaka formula \eqref{HN} allows to express the difference $H- H_N$ as the difference of the kernel $H$ between two consecutive discretization times plus a remainder term $\mathscr{R}_N$. We then study the weak approximation rate induced by $\mathscr{R}_N$ in Lemma \ref{remainder:HminusHN}. The main difficulty lies in the non-differentiability of $x\mapsto \partial_x\tilde{p}(t_j,t_{j'},x,x')$ at zero caused by the presence of the local time part in the dynamics of $X^{N}$ and $\tilde{X}^{N}$ which prevents us to use (again) the It-Tanaka formula. In a second step, in order to study $p-p^{d}_N$, we express this difference as an infinite parametrix series that involves the difference between the two convolution operators $\otimes$ and $\otimes_N$. Then, we notably use a (kind of) Lipschitz property in time for the transition density $p$ and a smoothing procedure for the drift part.

\vspace*{.3cm}
\noindent $\bullet$ \emph{Step 1: Error bound on $p^{d}_N-p_N$} \
\vspace*{.5cm}

Our first step consists in comparing $p_N$ with $p^{d}_N$. We first remark that for $r\geq1$
$$ 
\tilde{p} \otimes_N H^{(r)} - \tilde{p} \otimes_N H^{(r)}_N = \left( \left(\tilde{p} \otimes_N H^{(r-1)}\right) \otimes_N \left( H-H_N \right) \right) + \left( \left(\tilde{p}\otimes_N H^{(r-1)} - \tilde{p} \otimes_N H^{(r-1)}_N \right) \otimes_N H_{N}\right)
$$

\noindent so that summing the previous identity from $r=1$ to $r=\infty$ yields
$$
p^{d}_{N} - p_{N} = p^{d}_{N} \otimes_N (H-H_N) + (p^{d}_N - p_{N})\otimes_N H_N.
$$

\noindent Finally, by induction, for $0\leq j < j' \leq N$, we get
\begin{equation}
\label{comparison:pnd:pn}
\forall (x,x') \in \R \times \R^{*}, Ê(p^{d}_N- p_N)(t_j,t_{j'},x,x') = \sum_{r\geq0} \left\{p^{d}_{N} \otimes_N (H-H_N) \right\} \otimes_N H^{(r)}_N(t_j,t_{j'},x,x').
\end{equation}

\begin{lmm}\label{param:iterate:density:scheme:diff:kernel:lemma}Under \A{HR} and \A{HE}, for all $0\leq j < j' \leq N$, for all $r \geq0$, for all $(x,x') \in \R \times \R^{*}$, one has
\begin{equation}
\label{param:iterate:density:scheme:diff:kernel}
 |\left\{p^{d}_{N} \otimes_N (H-H_N) \right\} \otimes_N H^{(r)}_N(t_j,t_{j'},x,x')| \leq  C^{r} h^{\frac{\eta}{2}} (t_{j'}-t_{j})^{r \frac{\eta}{2}} \prod_{i=1}^{r} B\left(1+\frac{(i-1)\eta}{2},\frac{\eta}{2}\right) g_{c(t_{j'}-t_j)}(x'-x)
\end{equation}

\noindent  for some constant $c:=c(\lambda,\eta) \geq1$ and a non decreasing positive function $T\mapsto C:=C(T,b,\sigma)$. 
\end{lmm}

\begin{proof} For $j'=j+1$, one has
$$
(H - H_N)(t_j,t_{j+1},y,z) =  \left(b(y) \partial_x \tilde{p}(t_{j},t_{j+1},y,z) + \frac{1}{2}(a(y)-a(z))\partial^{2}_x \tilde{p}(t_j,t_{j+1},y,z) \right) - h^{-1} (p_{N}-\tilde{p}_N)(t_{j},t_{j+1},y,z).
$$

By the proof of Lemma \ref{param:iterate:density:scheme:lemma}, one gets
$$
h^{-1} (p_{N}-\tilde{p}_N(t_{j},t_{j+1},y,z)) \leq C(T,b,\sigma) \frac{h}{(t_{j'}-t_j)^{2-\frac{\eta}{2}}} g_{c(t_{j'}-t_j)}(z-y)
$$

\noindent and, similarly, by \eqref{kernel:bound} 
$$
\left|b(y) \partial_x \tilde{p}(t_{j},t_{j+1},y,z) + \frac{1}{2}(a(y)-a(z))\partial^{2}_x \tilde{p}(t_j,t_{j+1},y,z) \right| \leq  C(|b|_{\infty}(t_{j'}-t_j)^{\frac{1-\eta}{2}} + 1) \frac{h}{(t_{j'}-t_j)^{2-\frac{\eta}{2}}} g_{c(t_{j'}-t_j)}(z-y)
$$

\noindent so that, combining these two estimates, one gets
\begin{equation*}
\forall (y,z)\in \R \times \R^{*}, \, \, |(H - H_N)(t_j,t_{j'},y,z)| \leq C \frac{h}{(t_{j'}-t_{j})^{2-\frac{\eta}{2}}}  g_{c(t_{j'}-t_j)}(z-y).
\end{equation*}

\noindent where $T\mapsto C:=C(T,b,\sigma)$ is a non-decreasing positive function. For $j'>j+1$, $(y,z)\in \R \times \R^{*}$, using \eqref{HN} we write
\begin{align*}
(H- H_N)(t_j,t_{j'},y,z) & = b(y) \partial_x(\tilde{p}(t_j,t_{j'},y,z) - \tilde{p}(t_j+h,t_{j'},y,z) ) + \frac12 (a(y)-a(z)) \partial^2_x (\tilde{p}(t_j,t_{j'},y,z) - \tilde{p}(t_j+h,t_{j'},y,z)) \\
& \ \ + h^{-1} \mathscr{R}_{N}(t_j,t_{j'},x,x') .
\end{align*}

We treat the first term appearing in the right-hand side of the previous equality. By the mean value theorem and standard computations (separating eventually in four cases as previously done), one has
\begin{align*}
|  \partial^2_x ( \tilde{p}(t_j+h,t_{j'},y,z) - \tilde{p}(t_j,t_{j'},y,z) ) | & \leq C h \int_0^1  \frac{1}{(t_{j'}-t_j - \lambda h)^2} g_{c(t_{j'}-t_j - \lambda h)}(z-y) d\lambda \\
& \leq C \frac{h}{(t_{j'}-t_j)^2} g_{c(t_{j'}-t_j)}(z-y) 
\end{align*}

\noindent where we used the inequality $(t_{j'}-t_j - \lambda h)^{-(2+\frac12)} \leq C (t_{j'}-t_j)^{-(2+\frac12)} $ for some positive constant $C$. Combining the previous bound with \A{HR} implies
$$
\left|  \frac12 (a(y)-a(z)) \partial^2_x ( \tilde{p}(t_j+h,t_{j'},y,z) - \tilde{p}(t_j,t_{j'},y,z) ) \right| \leq C \frac{h}{(t_{j'}-t_j)^{2-\frac{\eta}{2}}} g_{c(t_{j'}-t_j)}(z-y).
$$

Similar computations show that 
$$
| b(y) \partial_x(\tilde{p}(t_j+h,t_{j'},y,z) - \tilde{p}(t_j,t_{j'},y,z) ) | \leq C|b|_\infty \frac{h}{(t_{j'}-t_j)^{\frac32}} g_{c(t_{j'}-t_j)}(z-y)
$$

\noindent which finally implies for $0\leq j < j' \leq N$ 
\begin{align}
\label{first:term:HminusHN}
| b(y) \partial_x(\tilde{p}(t_j,t_{j'},y,z) - \tilde{p}(t_j+h,t_{j'},y,z) ) & + \frac12 (a(y)-a(z)) \partial^2_x (\tilde{p}(t_j,t_{j'},y,z) - \tilde{p}(t_j+h,t_{j'},y,z))| \nonumber\\
& \leq C(|b|_{\infty}(t_{j'}-t_j)^{\frac{1-\eta}{2}}+ 1)  \frac{h}{(t_{j'}-t_{j})^{2-\frac{\eta}{2}}}  g_{c(t_{j'}-t_j)}(z-y).
\end{align}

We will now use the following decomposition
\begin{align}
\forall i \in \leftB 2, N \rightB, \quad p^{d}_{N} \otimes_N (H-H_N) (0,t_i, x,z) & = \sum_{k=0}^{i - 2} h \int_\R du  p^{d}_N(0,t_k,x,u)   (H-H_N)(t_k,t_{i},u,z) \nonumber\\
& \quad + h \int_{\R} p_N^{d}(0,t_{i-1},x,u) (H-H_N)(t_{i-1},t_i, u,z) du. \label{decomposition:first:convolution} 
\end{align}

From \eqref{gaussian:upper:bound:pnd}, \eqref{first:term:HminusHN} and the semigroup property, one gets
\begin{align}
h| \int_{\R} p_N^{d}(0,t_{i-1},x,u) (H-H_N)(t_{i-1},t_i, u,z) du| & \leq C  \frac{h^2}{h^{2-\frac{\eta}{2}}} \int_\R g_{c t_{i-1}}(u-x) g_{ch}(z-u) du \nonumber \\
& \leq C h^{\frac{\eta}{2}} g_{ct_i}(z-x)\label{estimate:border:term}
\end{align}

\noindent where $T \mapsto C:=C(T,b,\sigma)$ is a positive non-decreasing function.

From \eqref{first:term:HminusHN}, Lemma \ref{remainder:HminusHN} and the semigroup property, one gets
\begin{align}
| \sum_{k=0}^{i - 2} h \int_\R du  p^{d}_N(0,t_k,x,u)   (H-H_N)(t_k,t_{i},u,z)  | & \leq C \sum_{k=0}^{i - 2}  \frac{h^2}{(t_i-t_k)^{2-\frac{\eta}{2}}} \int_\R g_{c t_k}(u-x) g_{c(t_i-t_k)}(z-u) du \nonumber \\
&  \quad + C  \sum_{k=0}^{i - 2} \int_\R \left(\frac{h^2}{(t_i-t_k)^{2-\frac{\eta}{2}}} + (2\alpha-1) \frac{h^2}{(t_i-t_k)^{\frac{3-\eta}{2}}} g_{ch}(u) \right) \nonumber \\
&  \quad \quad\times g_{c t_k}(u-x)  g_{c(t_i-t_k)}(z-u) du \nonumber \\
& \leq C h^{\frac{\eta}{2}} \left(\sum_{k=2}^{i} \frac{1}{k^{2-\frac{\eta}{2}}}\right) g_{c t_i}(z-x) \nonumber \\
& \quad + C (2\alpha-1) h^{2} \sum_{k=0}^{i - 2} \frac{1}{(t_i-t_k)^{\frac{3-\eta}{2}}}  \int_\R g_{ch}(u) g_{c t_k}(u-x)  g_{c(t_i-t_k)}(z-u) du\label{temporary:bound:convolution}
\end{align}

\noindent where $T \mapsto C:=C(T,b,\sigma)$ is a non-decreasing positive function. Note that for $\eta<1$, the inequality $g_{ch}(u) \leq (2\pi c)^{-1/2} h^{-1/2}$ allows to achieve the rate $h^{\eta/2}$ and concludes the proof. However, if $\eta=1$, due to the local time part (that is the case  $\alpha\neq 1/2$), this bound provides an error of order $h^{1/2} |\log(h)|$ which is slightly worse than the announced rate. In order deal with this issue, we use the Cauchy-Schwarz inequality and the semigroup property in order to write 
\begin{align*}
\int_\R g_{ch}(u) g_{c t_k}(u-x)  g_{c(t_i-t_k)}(z-u) du & \leq \left(\int_\R g_{c t_k}(u-x) (g_{ch}(u))^2 du \right)^{\frac12} \left(\int_\R g_{c t_k}(u-x) (g_{c(t_i -t_k)}(z-u))^2 du \right)^{\frac12} \\
& \leq  C \frac{1}{h^{\frac14}(t_i-t_k)^{\frac14}} (g_{c t_{k+1}}(x))^{\frac12} (g_{c t_i}(z-x))^{\frac12} \\
& \leq C \frac{t^{\frac14}_i}{h^{\frac14}t^{\frac14}_{k+1}(t_i-t_k)^{\frac14}} g_{c t_i}(z-x).
\end{align*} 

Plugging the last bound into the second term appearing in the right hand side of \eqref{temporary:bound:convolution} yields
$$
h^{2} \sum_{k=0}^{i - 2} \frac{1}{(t_i-t_k)^{\frac{3-\eta}{2}}}  \int_\R g_{ch}(u) g_{c t_k}(u-x)  g_{c(t_i-t_k)}(z-u) du \leq C h^{2-\frac14} \left(\sum_{k=0}^{i-2} \frac{t^{\frac14}_i}{(t_i-t_k)^{\frac{7-2\eta}{4}}t^{\frac14}_{k+1}} \right) g_{c t_i}(z-x).
$$

For $0\leq k\leq \lfloor i/2 \rfloor $, one has $C^{-1} t_i \leq t_i - t_k \leq C t_i$ so that $h^{2-\frac14} \sum_{k=0}^{\lfloor i/2\rfloor} t^{\frac14}_i t^{-\frac14}_{k+1} (t_i-t_k)^{-\frac{(7-2\eta)}{4}} \leq C h^{2-\frac14} t^{-\frac{(3-\eta)}{2}}_i \sum_{k=0}^{\lfloor i/2 \rfloor} t^{-\frac14}_{k+1} \leq C h^{\frac34} t^{-\frac34+\frac{\eta}{2}}_i \leq C h^{\frac{\eta}{2}}$ for some positive constant $C>1$ that does not depend on $i$. For $\lfloor i/2\rfloor< k \leq i-1$, one has $C^{-1} t_i \leq t_{k} \leq C t_i $ so that $h^{2-\frac14} \sum_{k=\lfloor i/2 \rfloor +1}^{i-2} t^{\frac14}_i t^{-\frac14}_{k+1} (t_i-t_k)^{-\frac{(7-2\eta)}{4}} \leq C h^{\frac{\eta}{2}} \sum_{k=1}^{i} k^{-\frac{(7-2\eta)}{4}} \leq C h^{\frac{\eta}{2}}$ for some positive constant $C>1$ that does not depend on $i$. Using the decomposition, $\sum_{k=0}^{i-2} \cdots = \sum_{k=0}^{\lfloor i/2 \rfloor} \cdots + \sum_{\lfloor i/2 \rfloor +1}^{i-2} \cdots $ allows to conclude that $h^{2-\frac14} \sum_{k=0}^{\lfloor i/2\rfloor} t^{\frac14}_i t^{-\frac14}_{k+1} (t_i-t_k)^{-\frac{(7-2\eta)}{4}} \leq C h^{\frac{\eta}{2}}$. From these computations, we come back to \eqref{temporary:bound:convolution} and obtain
$$
| \sum_{k=0}^{i - 2} h \int_\R du  p^{d}_N(0,t_k,x,u)   (H-H_N)(t_k,t_{i},u,z)  |  \leq C h^{\frac{\eta}{2}} g_{c t_i}(z-x).
$$

Combining all the previous estimates, we derive
$$
\forall i \in \leftB 2, N \rightB, \quad |p^{d}_{N} \otimes_N (H-H_N) (0,t_i, x,z)| \leq C h^{\eta/2} g_{c  t_i}(z-x)
$$

\noindent where $T \mapsto C:=C(T,b,\sigma)$ is a non-decreasing positive function and \eqref{param:iterate:density:scheme:diff:kernel} follows easily by induction.
\end{proof}

From Lemma \ref{param:iterate:density:scheme:diff:kernel:lemma}, one clearly obtains
\begin{equation}
\label{pnd:pn:bound}
\forall (x,x') \in \R \times \R^{*}, Ê\left| (p^{d}_N- p_N)(t_j,t_{j'},x,x') \right| \leq C(T,b,\sigma) h^{\eta/2} g_{c (t_{j'}-t_j)}(x'-x). 
\end{equation}

\noindent $\bullet$ \emph{Step 2: Error bound on $p-p^{d}_N$} \ 
\vspace*{.5cm}

Now, to complete the proof of Theorem \ref{weak:error:scheme}, it remains to compare $p$ with $p^{d}_N$. For $r\geq1$, we write
\begin{align*}
\tilde{p} \otimes H^{(r)} - \tilde{p}\otimes_N H^{(r)}& = \left\{ \left(\tilde{p} \otimes H^{(r-1)}\right) \otimes H - \left(\tilde{p} \otimes H^{(r-1)}\right) \otimes_N H \right\}  \\
& + \left\{ \left( \tilde{p} \otimes H^{(r-1)} \right) - \left(\tilde{p} \otimes_N H^{(r-1)}\right) \right\} \otimes_N H
\end{align*}

\noindent so that summing the previous equality from $r=1$ and to infinity, we get
$$
p - p^{d}_N = p \otimes H - p \otimes_N H + (p-p^{d}_N) \otimes_N H.
$$ 

By induction of the previous identity, we obtain the uniformly and absolutely convergent series
\begin{equation}
\label{diff:p:and:pnd}
p - p^{d}_N = \sum_{r\geq 0} \left\{ p \otimes H - p \otimes_N H \right\} \otimes_N H^{(r)}.
\end{equation}

The rest of the proof is devoted to the study of \eqref{diff:p:and:pnd}. By the very definition of the continuous and discrete time space convolution, one has
\begin{align}
\left(p \otimes H - p \otimes_N H \right)(t_j,t_i,x,y) & = \sum_{k=j}^{i-1} \int_{t_k}^{t_{k+1}} \int_{\R} \left\{ p(t_{j},s,x,z) H(s,t_{i},z,y) - p(t_{j},t_k,x,z) H(t_k,t_{i},z,y) \right\} dz ds \nonumber \\
& =  \sum_{k=j}^{i-1} \int_{t_k}^{t_{k+1}} \int_{\R} (p(t_{j},s,x,z) - p(t_j, t_k,,x,z))  H(s,t_{i},z,y)  dz ds \label{first:term:decomposition:pnd} \\
& + \sum_{k=j}^{i-1} \int_{t_k}^{t_{k+1}} \int_{\R}  p(t_j, t_{k}, x, z) (H(s,t_{i},z,y) - H(t_k,t_{i},z,y)) dz ds  \label{second:term:decomposition:pnd}.
\end{align}

We first consider \eqref{first:term:decomposition:pnd} and separate the computations in three terms: 
\begin{align*}
A_1 & =  \int_{t_j}^{t_{j+1}} \left\{ (\int_{\R} p(t_{j},s,x,z) H(s,t_{i},z,y)dz) -  H(s,t_i,x,y) \right\} ds,  \\
A_2 & = \sum_{j+1 \leq k \leq (i+j)/2} \int_{t_k}^{t_{k+1}} \int_{\R} (p(t_{j},s,x,z) - p(t_j, t_k,x,z))  H(s,t_{i},z,y)  dz ds, \\
A_3 & = \sum_{(i+j)/2 < k \leq (i-1)} \int_{t_k}^{t_{k+1}} \int_{\R} (p(t_{j},s,x,z) - p(t_j, t_k,x,z))  H(s,t_{i},z,y)  dz ds.
\end{align*}

First, by the semigroup relation and standard computations, one has
$$
|A_1| \leq C \left\{ \int_{t_j}^{t_{j+1}}  \frac{1}{(t_i-s)^{1-\frac{\eta}{2}}}ds  \right\} g_{c (t_i-t_j)}(y-z) \leq C \frac{h}{(t_i-t_j)^{1-\frac{\eta}{2}}} g_{c (t_i-t_j)}(y-z).
$$

Now, using Lemma \ref{Lipschitz:param:series:skew:diff}, the semigroup property and the inequality $(t_{k}-t_j)^{-1} \leq 2 (t_i -t_{j})^{-1}$ for $j+1 \leq k \leq (i+j)/2$, one gets
\begin{align*}
|A_3| & \leq C \frac{h}{t_i-t_j}  \left(\sum_{j+1 \leq k \leq (i-1)} \int_{t_k}^{t_{k+1}} \frac{1}{(t_i-s)^{1-\frac{\eta}{2}}} ds \right) g_{c (t_i-t_j)}(y-x) \\
& \leq  C  \frac{h}{t_i-t_j} \left( \int_{t_j}^{t_{i}} \frac{1}{(t_i-s)^{1-\frac{\eta}{2}}} ds \right) g_{c (t_i-t_j)}(y-x)  \\
& \leq C \frac{h}{(t_i-t_j)^{1-\frac{\eta}{2}}} g_{c (t_i-t_j)}(y-x) .
\end{align*}

In order to deal with $A_2$, we proceed using a regularization argument in order to obtain the differentiability of $t \mapsto p(0,t,x,y)$ on $[0,T]$ for fixed $(x,y) \in \R \times \R^{*}$. Let us note that this differentiability is not guaranteed under \A{HR} and \A{HE}. From Theorem 174 p.111 of Kestelman \cite{kestelman}, there exists a sequence $(b_N)_{N\in \N^{*}}$ of continuous functions such that:
\begin{equation}
\lim_{N\rightarrow \infty} b_N = b, \, \, a.e. \, \mbox{ and } \, \sup_{N\in \N^{*}} |b_N|_{\infty} \leq |b|_{\infty}.
\end{equation}

We also consider a positive mollifier $\rho$ on $\R$ and write $b^{\varepsilon}_N(x)  = \int_{\R} \varepsilon^{-1} \rho((x-y)/\varepsilon) b_N(y) dy$. For fixed $N$ and $\varepsilon$, $b^{\varepsilon}_{N}$ is smooth and satisfies: for all $x\in \R$, $\lim_{\varepsilon \rightarrow 0} b^{\varepsilon}_N(x) = b_N(x)$, $\sup_{\varepsilon, N}|b^{\varepsilon}_N|_{\infty} \leq |b|_\infty$.
 
 We denote by $p^{\varepsilon,N}$ the transition density of the skew diffusion $X^{\varepsilon,N}$ obtained by replacing the drift $b$ in dynamics \eqref{skew:diff} by $b^{\varepsilon}_N$. From the parametrix expansion \eqref{repres:formula:param} obtained in \cite{kohatsu:taguchi:zhong}, one has
\begin{equation}
\label{regularized:param:series}
p^{\varepsilon,N}(0,t,x,y) = \sum_{r\geq0}( \tilde{p} \otimes H^{(r)}_{\varepsilon,N})(0,t,x,y)
\end{equation}

\noindent where $H_{\varepsilon,N}(s,t,x,y) = b^{\varepsilon}_N(x) \partial_x \tilde{p}(s,t,x,y) + \frac{(a(x)-a(y))}{2} \partial^{2}_x \tilde{p}(s,t,x,y)$. Moreover, the series is again absolutely and uniformly convergent $(0,T] \times \R \times \R^{*}$ and letting $\varepsilon$ goes to zero and then $N$ goes to infinity, from the dominated convergence theorem applied to the series \eqref{regularized:param:series}, one gets $p^{\varepsilon,N}(0,t,x,y) \rightarrow p(0,t,x,y)$ and $p^{\varepsilon,N}$ satisfies the Gaussian upper-bound:
$$
p^{\varepsilon,N}(0,t,x,y) \leq E_{\eta/2,1}(C(T^{(\frac{(3-2\eta)}{2}}|b^{\varepsilon}_N|_{\infty}+T^{1-\frac{\eta}{2}})) g_{c t}(y-x) \leq E_{\eta/2,1}(C(T^{(\frac{(3-2\eta)}{2}}|b|_{\infty}+T^{1-\frac{\eta}{2}})) g_{c t}(y-x).
$$

\noindent with constants $C(\lambda,\eta), c(\lambda,\eta)>1$ independent of $N$ and $\varepsilon$. By similar arguments as those employed in Chapter 1, \cite{friedman:64}, one derives from \eqref{regularized:param:series} that $t\mapsto p^{\varepsilon}_N(0,t,x,y)$ is continuously differentiable on $(0,T]$ for fixed $(x,y) \in \R \times \R^*$. Now, from Lemma \ref{Lipschitz:param:series:skew:diff} noting that $|b^{\varepsilon}_N|_{\infty} \leq |b|_{\infty}$, it follows that for all $t\in (0,T]$
\begin{equation}
\label{Lipschitz:bound:regularized:density}
 | \partial_t p^{\varepsilon,N}(0,t,x,y) | =  \lim_{h \rightarrow 0} | \frac{p^{\varepsilon,N}(0,t+h,x,y) - p^{\varepsilon,N}(0,t,x,y)}{h} | \leq C t^{-1} g_{c t}(y-x)
\end{equation}

\noindent where $C = C(\lambda,\eta)(|b|_{\infty}T^{\frac{1-\eta}{2}}+1) E_{\eta/2,1}(C(T^{(\frac{(3-2\eta)}{2}}|b|_{\infty}+T^{1-\frac{\eta}{2}}))$, $C(\lambda,\eta),\ c:=c(\lambda,\eta)>1$ do not depend on $N$ and $\varepsilon$. Moreover, still by dominated convergence theorem, letting $\varepsilon \rightarrow 0$ then $N\rightarrow +\infty$, one has
$$
A^{\varepsilon,N}_2 := \sum_{j+1 \leq k \leq (i+j)/2} \int_{t_k}^{t_{k+1}} \int_{\R} (p^{\varepsilon,N}(t_{j},s,x,z) - p^{\varepsilon,N}(t_j, t_k,,x,z))  H(s,t_{i},z,y)  dz ds \rightarrow A_2.
$$

 Now, by \eqref{Lipschitz:bound:regularized:density}, the semigroup property and the inequality $(t_i-t_{k})^{-(1-\frac{\eta}{2})} \leq C (t_i -t_{j})^{-(1-\frac{\eta}{2})}$ valid for $j+1 \leq k \leq (i+j)/2$, one gets for a constant $C$ that does not depend on $\varepsilon$
\begin{align*}
|A^{\varepsilon,N}_2| & \leq \frac{C}{(t_i-t_j)^{1-\frac{\eta}{2}}}  \left(\sum_{j+1 \leq k \leq (i+j)/2} \int_{t_k}^{t_{k+1}} \int_\R \int_{t_k}^{s} |\partial_t p^{\varepsilon,N}(0,t-t_j,x,z)| dt \right) g_{c (t_i-s)}(y-z) dz ds \\
& \leq  \frac{C}{(t_i-t_j)^{1-\frac{\eta}{2}}} \left(\sum_{j+1 \leq k \leq (i+j)/2} \int_{t_k}^{t_{k+1}} \log\left(\frac{s-t_j}{t_k-t_j}\right) ds \right)  g_{c (t_i-t_j)}(y-x) \\
& =  \frac{C h}{(t_i-t_j)^{1-\frac{\eta}{2}}}  \left(\sum_{j+1 \leq k \leq  i-1} \left\{ (k+1-j) \log\left( \frac{k+1-j}{k-j}\right) - 1 \right\} \right)   g_{c  (t_i-t_j)}(y-x) \\
& = \frac{C h}{(t_i-t_j)^{1-\frac{\eta}{2}}}  \left(\sum_{1 \leq k \leq  i-j-1} \left\{ (k+1) \log\left( \frac{k+1}{k}\right) - 1 \right\} \right)   g_{c (t_i-t_j)}(y-x) \\
& = \frac{C h}{(t_i-t_j)^{1-\frac{\eta}{2}}} g_{c (t_i-t_j)}(y-x)
\end{align*}
 
 \noindent where we used a second order Taylor expansion for the last but one equality. Hence, letting $\varepsilon \rightarrow 0$ and then $N\rightarrow \infty$ in the previous inequality yields
 $$
 |A_2| \leq  \frac{C h}{(t_i-t_j)^{1-\frac{\eta}{2}}}  \left(\sum_{1 \leq k \leq  i-j-1}  \left\{ (k+1) \log\left( \frac{k+1}{k}\right) - 1 \right\} \right)   g_{c (t_i-t_j)}(y-x).
 $$


We now consider \eqref{second:term:decomposition:pnd}. One has 
$$
|H(s,t_{i},z,y) - H(t_k,t_{i},z,y) | =  | \int_{t_k}^{s} \partial_t H(t,t_i,z,y) dt | \leq C \left( \int_{t_k}^{s} (t_i-t)^{-(2- \frac{\eta}{2})} dt \right)g_{c (t_i-t_k)} (y-z) 
$$

\noindent which, by the semigroup relation, clearly yields
\begin{align*}
\int_{t_k}^{t_{k+1}} \int_{\R}  p(t_j, t_{k}, x, z) (H(s,t_{i},z,y) - H(t_k,t_{i},z,y)) dz ds & \leq C \left( \int_{t_k}^{t_{k+1}} \left(\frac{1}{(t_i-s)^{1-\frac{\eta}{2}}} - \frac{1}{(t_i-t_k)^{1-\frac{\eta}{2}}} \right) ds \right) g_{c (t_i-t_j)} (y-x) \\
& = C \left( \frac{2}{\eta} \left((t_i-t_k)^{\frac{\eta}{2}} - (t_i-t_{k+1})^{\frac{\eta}{2}}\right) - \frac{h}{(t_i-t_k)^{1-\frac{\eta}{2}}} \right) g_{c (t_i-t_j)} (y-x).
\end{align*}

Again, by a second order Taylor expansion, one easily gets
\begin{align*}
\sum_{k=j}^{i-1} \frac{2}{\eta}\left((t_i-t_k)^{\frac{\eta}{2}} - (t_i-t_{k+1})^{\frac{\eta}{2}} \right)- \frac{h}{(t_i-t_k)^{1-\frac{\eta}{2}}}& = \sum_{k=1}^{i-j} \left\{ \frac{2}{\eta} (t_k^{\frac{\eta}{2}} - t_{k-1}^{\frac{\eta}{2}}) - \frac{h}{t^{1- \frac{\eta}{2}}_{k}}   \right\} \\
& = h^{\eta/2} \sum_{k=0}^{i-j-1} \left\{ \frac{2}{\eta} ( (k+1)^{\frac{\eta}{2}} - k^{\frac{\eta}{2}}) - \frac{1}{(k+1)^{1- \frac{\eta}{2}}}   \right\} \\
&  \leq C h^{\eta/2}
\end{align*}

\noindent which in turn yields
$$
| \sum_{k=j}^{i-1} \int_{t_k}^{t_{k+1}} \int_{\R}  p(t_j, t_{k}, x, z) (H(s,t_{i},z,y) - H(t_k,t_{i},z,y)) dz ds | \leq C h^{\eta/2} g_{c (t_i-t_j)} (y-x).
$$

Combining the two estimates we finally get
$$
\left| \left(p \otimes H - p \otimes_N H \right)(t_j,t_i,x,y) \right| \leq C(T,b,\sigma) h^{\frac{\eta}{2}} g_{c (t_i-t_j)}(y-x)
$$

\noindent where $T \mapsto C(T,b,\sigma)$ is a non-decreasing function. As a consequence of the previous computations, one also obtains
\begin{align*}
| \left(p \otimes H - p \otimes_N H \right)\otimes_N H(t_j,t_i,x,y) | & \leq C^2 h^{\eta/2} \left(\sum_{k=j+1}^{i-1} h  \frac{1}{(t_i-t_k)^{1- \frac{\eta}{2}}}\right) g_{c(t_i -t_j)}(y-x) \\
& \leq  C^2 (t_i-t_j)^{\eta/2}  h^{\eta/2} B(1, \frac{\eta}{2})  g_{c (t_i -t_j)}(y-x)
\end{align*}

\noindent and, by induction, for $r\geq 0$, one gets
\begin{align*}
| \left(p \otimes H - p \otimes_N H \right)\otimes_N H^{(r)}(t_j,t_i,x,y) | & \leq  C^{r+1}  h^{\eta/2} (t_i-t_j)^{ r \frac{\eta}{2}} \prod_{i=1}^{r}B\left(1 + (i-1)\frac{\eta}{2}, \frac{\eta}{2}\right)  g_{c (t_i -t_j)}(y-x).
\end{align*}

We plug the previous bound in \eqref{diff:p:and:pnd}. The asymptotic of the Gamma function readily yields the convergence of the series as well as a Gaussian upper-bound, namely, one gets
$$
| (p - p^{d}_N)(t_j,t_i,x,y)| \leq C(T,b,\sigma,\eta) h^{\eta/2} g_{c (t_i -t_j)}(y-x).
$$ 

Combining the previous bound with \eqref{pnd:pn:bound} completes the proof of Theorem \ref{weak:error:scheme}.

\section{Appendix: Proof of some technical results}\label{Appendix}

\begin{lem}\label{remainder:HminusHN} Under \A{HR} and \A{HE}, there exist a constant $c:=c(\lambda,\eta) \geq1$ such that for $(j, j') \in \left\{0,\cdots,N\right\}^2$ with $j+1 < j'$ and for all $(x,x') \in  (\R^{*})^2$, one has
\begin{equation*}
|\mathscr{R}_{N}(t_j,t_{j'},x,x')| \leq C \left(\frac{h^2}{(t_{j'}-t_j)^{2-\frac{\eta}{2}}} + (2\alpha-1) \frac{h^2}{(t_{j'}-t_j)^{\frac{3-\eta}{2}}} g_{ch}(x) \right) g_{c(t_{j'}-t_j)}(x'-x)
\end{equation*}

\noindent where $T\mapsto C=C(T,b, \sigma)$ is a positive non-decreasing function.
\end{lem}

\begin{proof} From \eqref{RN:ito:tanaka}, we use the following decomposition
$$
\mathscr{R}_{N}(t_j,t_{j'},x,x') = \mathscr{R}^1_{N}(t_j,t_{j'},x,x') + \mathscr{R}^2_{N}(t_j,t_{j'},x,x') + \mathscr{R}^3_{N}(t_j,t_{j'},x,x')
$$

\noindent with
\begin{align*}
 \mathscr{R}^1_{N}(t_j,t_{j'},x,x') & :=  b(x) \int_0^h \left\{ \E[\partial_x\tilde{p}(t_{j+1},t_{j'},X^{N,t_{j},x}_{s},x')] - \partial_x\tilde{p}(t_{j+1},t_{j'},x,x') \right\} ds, \\
 \mathscr{R}^2_{N}(t_j,t_{j'},x,x') & := \frac12 (a(x)-a(x')) \int_0^h \left\{ \E[\partial^2_x\tilde{p}(t_{j+1},t_{j'},X^{N,t_{j},x}_{s},x')] - \partial^2_x\tilde{p}(t_{j+1},t_{j'},x,x') \right\} ds,  \\
 \mathscr{R}^3_{N}(t_j,t_{j'},x,x') & := - \frac12 a(x') \int_0^h \left\{ \E[\partial^2_x\tilde{p}(t_{j+1},t_{j'},\tilde{X}^{N,t_{j},x}_{s},x')] - \E[\partial^2_x\tilde{p}(t_{j+1},t_{j'},X^{N,t_{j},x}_{s},x')] \right\} ds .
\end{align*}

The main difficulty lies in the presence of the local time part in the dynamics of $X^{N}$ and $\tilde{X}^N$. More precisely, we remark that if $\alpha \neq 1/2$, $z\mapsto \partial_x\tilde{p}(t_{j+1},t_{j'},z,x'), \ \partial^2_x\tilde{p}(t_{j+1},t_{j'},z,x')$ are not continuous and (consequently) not differentiable at $0$ so we cannot apply the It formula. We instead rely on Taylor's expansion. For the first term $\mathscr{R}^1_{N}(t_j,t_{j'},x,x')$ we use the decomposition
\begin{align*}
\E[\partial_x\tilde{p}(t_{j+1},t_{j'},X^{N,t_{j},x}_{s},x')] & - \partial_x\tilde{p}(t_{j+1},t_{j'},x,x')   = \int_{\R} p_N(0,s,x,z) (\partial_x\tilde{p}(t_{j+1},t_{j'},z,x') - \partial_x\tilde{p}(t_{j+1},t_{j'},x,x')) dz \\
& = \int_{\R  \backslash \left\{z: |z-x|^2 < (t_{j'}-t_{j+1})\right\}} p_N(0,s,x,z) (\partial_x\tilde{p}(t_{j+1},t_{j'},z,x') - \partial_x\tilde{p}(t_{j+1},t_{j'},x,x') \\
& \quad - \partial^2_x\tilde{p}(t_{j+1},t_{j'},x,x') (z-x)) dz \\
& \quad + \int_{\left\{z: |z-x|^2 < (t_{j'}-t_{j+1})\right\}} p_N(0,s,x,z) (\partial_x\tilde{p}(t_{j+1},t_{j'},z,x') - \partial_x\tilde{p}(t_{j+1},t_{j'},x,x') \\
& \quad - \partial^2_x\tilde{p}(t_{j+1},t_{j'},x,x') (z-x)) dz  \\
& \quad + \left( \int_{\R} p_N(0,s,x,z)   (z-x) dz\right) \partial^2_x\tilde{p}(t_{j+1},t_{j'},x,x').
\end{align*}

 If $|z-x|^2 \geq t_{j'}-t_{j+1}$,  one has $|\partial_x \tilde{p}(t_{j+1},t_{j'},z,x')|\leq C (t_{j'}-t_{j+1})^{-1/2} g_{c (t_{j'}-t_{j+1}}(x'-z) \leq C |z-x|^2 (t_{j'}-t_{j+1})^{-3/2} g_{c (t_{j'}-t_{j+1})}(x'-z)$ and similarly one gets $|\partial_x \tilde{p}(t_{j+1},t_{j'},x, x')| \leq C  |z-x|^2 (t_{j'}-t_{j+1})^{-3/2} g_{c (t_{j'}-t_{j+1})}(x'-x)$ and $|\partial^2_x \tilde{p}(t_{j+1},t_{j'},x, x')| \leq C  |z-x| (t_{j'}-t_{j+1})^{-3/2} g_{c (t_{j'}-t_{j+1})}(x'-x)$ for some positive constant $C:=C(\lambda)>1$. Combining the latter bounds with \eqref{bound:transition:pN} yields
\begin{align}
|\int_{\R  \backslash \left\{z: |z-x|^2 < (t_{j'}-t_{j+1})\right\}} p_N(0,s,x,z)&  (\partial_x\tilde{p}(t_{j+1},t_{j'},z,x') - \partial_x\tilde{p}(t_{j+1},t_{j'},x,x')  - \partial^2_x\tilde{p}(t_{j+1},t_{j'},x,x') (z-x)) dz | \nonumber \\
& \leq C (t_{j'}-t_{j+1})^{-\frac32} \int_{\R} g_{c s}(z-x) |z-x|^2 (g_{c (t_{j'}-t_{j+1})}(x'-x) + g_{c (t_{j'}-t_{j+1})}(z-x')) dz \nonumber \\
& \leq C \frac{s}{(t_{j'}-t_{j+1})^{\frac32}}  \left\{g_{c (t_{j'}-t_{j+1})}(x'-x)  + g_{c (t_{j'}-t_{j+1}+s)}(x'-x)  \right\} \nonumber \\
& \leq C \frac{s}{(t_{j'}-t_{j})^{\frac32}} g_{c (t_{j'}-t_{j})}(x'-x) \label{case1}
\end{align}

\noindent where $T \mapsto C:=C(T,b,\sigma)$ is a positive non-decreasing function and where we used that $j'>j+1$ for the last inequality.

 If $|z-x|^2 < (t_{j'}-t_{j+1})$, we remark that $ \mathscr{D} = \left\{ z: z\mapsto \partial_x\tilde{p}(t_{j+1},t_{j'},z,x') \mbox{ is not twice differentiable}  \right\}=\left\{ 0\right\}$, so that by Taylor's expansion, one gets
\begin{align}
\int_{\left\{z: |z-x|^2 < (t_{j'}-t_{j+1})\right\}} p_N(0,s,x,z) & \left( \partial_x \tilde{p}(t_{j+1},t_{j'},z,x') - \partial_x \tilde{p}(t_{j+1},t_{j'},x, x') -  \partial^2_x \tilde{p}(t_{j+1},t_{j'},x, x') (z-x) \right) dz  \nonumber \\
& =  \int_{\left\{z: |z-x|^2 < (t_{j'}-t_{j+1})\right\}} p_N(0,s,x,z) \frac12 \partial^3_x \tilde{p}(t_{j+1},t_{j'},\zeta, x') (z-x)^2 dz \nonumber
\end{align}

\noindent where $\zeta$ is a point in $(z,x)$. Observe that since $|z-x|^2 < (t_{j'}-t_{j+1})$ for any point $\zeta \in (z,x)$ one has
$$
\exp\left(- \frac{(\zeta-x')^2}{c(t_{j'}-t_{j+1})}\right) \leq C \left( \exp\left(-\frac{(z-x')^2}{c(t_{j'}-t_{j+1})}\right) + \exp\left(-\frac{(x-x')^2}{c(t_{j'}-t_{j+1})}\right) \right).
$$

\noindent Combining the previous inequality with \eqref{bound:transition:pN} and the semigroup property we obtain
\begin{align}
 | \int_{\left\{z: |z-x|^2 < (t_{j'}-t_{j+1})\right\}} p_N(0,s,x,z) \frac12 \partial^3_x \tilde{p}(t_{j+1},t_{j'},\zeta, x') (z-x)^2 dz | & \leq C  \frac{s}{(t_{j'}-t_{j})^{\frac32}}  g_{c (t_{j'}-t_{j})}(x'-x) \label{taylor:expansion:rn1}
\end{align}

\noindent where $T \mapsto C:=C(T,b,\sigma)$ is a positive non-decreasing function and where we again used  that $j'>j+1$.
We now consider the quantity $\left( \int_{\R} p_N(0,s,x,z)   (z-x) dz\right) \partial^2_x\tilde{p}(t_{j+1},t_{j'},x,x') = (\E[X^{N}_s-x]) \partial^2_x\tilde{p}(t_{j+1},t_{j'},x,x')$. Using the dynamics \eqref{Euler:scheme}, we get $\E[X^{N}_s-x] = b(x) s + (2\alpha-1) \E[L^0_s(X^N)]$. By dominated convergence theorem and the gaussian upper-estimate \eqref{bound:transition:pN} satisfied by $p_N$, one has 
\begin{align*}
\E[L^0_s(X^N)] & = \E[\lim_{\eta \rightarrow 0} \frac{1}{2\eta}\int_0^s \mathbb{I}_{\left\{-\eta \leq X_u \leq \eta\right\}} a(x)du] \\
& = a(x) \lim_{\eta \rightarrow 0} \int_0^s \frac{\P(X_u\leq \eta)-\P(X_u\leq -\eta)}{2 \eta} du \\
& = a(x) \int_0^s \frac{p_{N}(0,u,x,0+) + p_{N}(0,u,x,0-)}{2} du. \\
& \leq C \int_0^s g_{cu}(x) du \\
& \leq C s g_{cs}(x)
\end{align*}

\noindent which in turn yields
\begin{align}
\label{3term:rn1}
|\left( \int_{\R} p_N(0,s,x,z)   (z-x) dz\right) \partial^2_x\tilde{p}(t_{j+1},t_{j'},x,x') | & \leq C \frac{s}{(t_{j'}-t_j)} (1+ (2\alpha-1)g_{cs}(x)) g_{c(t_{j'}-t_{j})}(x'-x)
\end{align}

\noindent where $T \mapsto C:=C(T,b,\sigma)$ is a non-decreasing positive function.
Combining \eqref{case1}, \eqref{taylor:expansion:rn1} and \eqref{3term:rn1} we obtain
\begin{align}
|\mathscr{R}^1_{N}(t_j,t_{j'},x,x') |&  \leq C \int_0^h \left(\frac{s}{(t_{j'}-t_{j})^{\frac32}} + (2\alpha-1)\frac{s }{(t_{j'}-t_{j})}g_{cs}(x)\right)  g_{c(t_{j'}-t_{j})}(x'-x) \nonumber \\
& \leq C  \left(\frac{h^2}{(t_{j'}-t_{j})^{\frac32}} + (2\alpha-1)\frac{h^2}{(t_{j'}-t_{j})}g_{ch}(x)\right)  g_{c(t_{j'}-t_{j})}(x'-x). \label{bound:rn1}
\end{align}

We now focus on the second remainder term $ \mathscr{R}^2_{N}(t_j,t_{j'},x,x')$. Similarly to the previous case, we use the decomposition
\begin{align*}
\E[\partial^2_x\tilde{p}(t_{j+1},t_{j'},X^{N,t_{j},x}_{s},x')] & - \partial^2_x\tilde{p}(t_{j+1},t_{j'},x,x')   = \int_{\R  \backslash \left\{z: |z-x|^2 < (t_{j'}-t_{j+1})\right\}} p_N(0,s,x,z) (\partial^2_x\tilde{p}(t_{j+1},t_{j'},z,x') - \partial^2_x\tilde{p}(t_{j+1},t_{j'},x,x') \\
& \quad - \partial^3_x\tilde{p}(t_{j+1},t_{j'},x,x') (z-x)) dz \\
& \quad + \int_{\left\{z: |z-x|^2 < (t_{j'}-t_{j+1})\right\}} p_N(0,s,x,z) (\partial^2_x\tilde{p}(t_{j+1},t_{j'},z,x') - \partial^2_x\tilde{p}(t_{j+1},t_{j'},x,x') \\
& \quad - \partial^3_x\tilde{p}(t_{j+1},t_{j'},x,x') (z-x)) dz  \\
& \quad + \left( \int_{\R} p_N(0,s,x,z)   (z-x) dz\right) \partial^3_x\tilde{p}(t_{j+1},t_{j'},x,x').
\end{align*}

Following similar lines of reasoning as for the first case, one successively gets
\begin{align}
| \int_{\R  \backslash \left\{z: |z-x|^2 < (t_{j'}-t_{j+1})\right\}} p_N(0,s,x,z) &  (\partial^2_x\tilde{p}(t_{j+1},t_{j'},z,x') - \partial^2_x\tilde{p}(t_{j+1},t_{j'},x,x') - \partial^3_x\tilde{p}(t_{j+1},t_{j'},x,x') (z-x)) dz | \nonumber \\
& \leq C \frac{s}{(t_{j'}-t_{j})^2} g_{c(t_{j'}-t_{j})}(x'-x) \label{bound1:rn2}
\end{align}

\noindent and, for some $\zeta \in (z,x)$,
\begin{align}
 | \int_{\left\{z: |z-x|^2 < (t_{j'}-t_{j+1})\right\}} p_N(0,s,x,z) & (\partial^2_x\tilde{p}(t_{j+1},t_{j'},z,x') - \partial^2_x\tilde{p}(t_{j+1},t_{j'},x,x') - \partial^3_x\tilde{p}(t_{j+1},t_{j'},x,x') (z-x)) dz | \nonumber \\
 & = |  \int_{\left\{z: |z-x|^2 < (t_{j'}-t_{j+1})\right\}} p_N(0,s,x,z) \frac12 \partial^4_x\tilde{p}(t_{j+1},t_{j'},\zeta,x') (z-x)^2 dz | \nonumber \\
 & \leq C \frac{s}{(t_{j'}-t_{j})^2} g_{c(t_{j'}-t_{j})}(x'-x)\label{bound2:rn2}
\end{align}

\noindent where $T \mapsto C:=C(T,b,\sigma)$ is a non-decreasing positive function. Finally, one also obtains
\begin{align}
|\left( \int_{\R} p_N(0,s,x,z)   (z-x) dz\right) \partial^3_x\tilde{p}(t_{j+1},t_{j'},x,x')| & \leq C \frac{s}{(t_{j'}-t_{j})^{\frac32}} (1+ (2\alpha-1)g_{cs}(x)) g_{c(t_{j'}-t_j)}(x'-x).\label{bound3:rn2}
\end{align}

Combining \eqref{bound1:rn2}, \eqref{bound2:rn2}, \eqref{bound3:rn2} with \A{HR} yields
\begin{align}
|\mathscr{R}^2_{N}(t_j,t_{j'},x,x') |& \leq C |x-x'|^\eta  \left( \int_0^h \left(\frac{s}{(t_{j'}-t_j)^2} + (2\alpha-1) \frac{s}{(t_{j'}-t_j)^{\frac32}}  g_{cs}(x)\right) ds \right) g_{c(t_{j'}-t_j)}(x'-x) \nonumber \\
& \leq C \left(\frac{h^2}{(t_{j'}-t_j)^{2-\frac{\eta}{2}}} + (2\alpha-1) \frac{h^2}{(t_{j'}-t_j)^{\frac{3-\eta}{2}}}  g_{ch}(x)\right) g_{c(t_{j'}-t_j)}(x'-x). \label{bound:rn2}
\end{align}

We now conclude by the third remainder term $ \mathscr{R}^3_{N}(t_j,t_{j'},x,x') $. We first remark that 
\begin{align*}
\E[\partial^2_x\tilde{p}(t_{j+1},t_{j'},\tilde{X}^{N,t_{j},x}_{s},x')]  - \E[\partial^2_x\tilde{p}(t_{j+1},t_{j'},X^{N,t_{j},x}_{s},x')]  & = \int_{\R} (\tilde{p}^{t_{j'},x'}_N(t_j,t_j+s,x,z)-p_N(0,s,x,z) ) \partial^2_x\tilde{p}(t_{j+1},t_{j'},z,x') dz \\
& =  \int_{\R} ( p^{x'}(0,s,x,z) - p_N(0,s,x,z) ) \partial^2_x\tilde{p}(t_{j+1},t_{j'},z,x') dz 
\end{align*}

\noindent and, similarly to the previous terms, we use the following decomposition
\begin{align}
\int_{\R} & ( p^{x'}(0,s,x,z) - p_N(0,s,x,z) )  \partial^2_x\tilde{p}(t_{j+1},t_{j'},z,x') dz = \left(\int_{\R} ( p^{x'}(0,s,x,z) - p_N(0,s,x,z) ) (z-x) dz\right) \partial^3_x\tilde{p}(t_{j+1},t_{j'},x,x')   \nonumber\\
&  + \int_{\left\{z: |z-x|^2 < (t_{j'}-t_{j+1})\right\}} ( p^{x'}(0,s,x,z) - p_N(0,s,x,z) ) (\partial^2_x\tilde{p}(t_{j+1},t_{j'},z,x') - \partial^2_x\tilde{p}(t_{j+1},t_{j'},x,x') - \partial^3_x\tilde{p}(t_{j+1},t_{j'},x,x')(z-x)) \nonumber \\
&  + \int_{\R \backslash \left\{z: |z-x|^2 < (t_{j'}-t_{j+1})\right\}} ( p^{x'}(0,s,x,z) - p_N(0,s,x,z) ) (\partial^2_x\tilde{p}(t_{j+1},t_{j'},z,x') - \partial^2_x\tilde{p}(t_{j+1},t_{j'},x,x') - \partial^3_x\tilde{p}(t_{j+1},t_{j'},x,x')(z-x)). \label{decomposition:rn3}
\end{align}

From \eqref{diff:p:tildep} and the mean value theorem, for some $\zeta \in (z,x)$ one gets
\begin{align}
\int_{\left\{z: |z-x|^2 < (t_{j'}-t_{j+1})\right\}} ( p^{x'}(0,s,x,z) - p_N(0,s,x,z) ) &  (\partial^2_x\tilde{p}(t_{j+1},t_{j'},z,x') - \partial^2_x\tilde{p}(t_{j+1},t_{j'},x,x') - \partial^3_x\tilde{p}(t_{j+1},t_{j'},x,x')(z-x)) \nonumber \\
& \leq C (|b|_{\infty} s^{\frac12} + |x-x'|^{\eta}) \int_{\left\{z: |z-x|^2 < (t_{j'}-t_{j+1})\right\}} g_{cs}(z-x) |z-x|^2 \nonumber \\
& \quad \times |\partial^4_x\tilde{p}(t_{j+1},t_{j'},\zeta,x')| dz \nonumber \\
& \leq C \frac{s}{(t_{j'}-t_j)^{2-\frac{\eta}{2}}} g_{c(t_{j'}-t_j)}(x'-x) \label{second:term:rn3}
\end{align}

\noindent and similarly to \eqref{case1}
\begin{align}
\int_{\R \backslash \left\{z: |z-x|^2 < (t_{j'}-t_{j+1})\right\}} ( p^{x'}(0,s,x,z) - p_N(0,s,x,z) ) & (\partial^2_x\tilde{p}(t_{j+1},t_{j'},z,x') - \partial^2_x\tilde{p}(t_{j+1},t_{j'},x,x') - \partial^3_x\tilde{p}(t_{j+1},t_{j'},x,x')(z-x)) \nonumber \\
& \leq C \frac{s}{(t_{j'}-t_j)^{2-\frac{\eta}{2}}} g_{c(t_{j'}-t_j)}(x'-x) \label{third:term:rn3}
\end{align}

\noindent where $T \mapsto C:=C(T,b,\sigma)$ is a non-decreasing positive function.
For the first term appearing in the right-hand side of \eqref{decomposition:rn3}, we write $\int_{\R} ( p^{x'}(0,s,x,z) - p_N(0,s,x,z) ) (z-x) dz = \E[X^{x'}_s-x] - \E[X^{N}_s-x] = (2\alpha-1)\E[L^{0}_s(X^{x'})] - b(x) s - (2\alpha-1) \E[L^{0}_s(X^{N})]$. By dominated convergence theorem, one gets 
\begin{align*}
\E[L^{0}_s(X^{x'})] - \E[L^{0}_s(X^{N})] & = a(x') \int_0^s \frac{p^{x'}(0,u,x,0+)+p^{x'}(0,u,x,0-)}{2} du - a(x) \int_0^s \frac{p_N(0,u,x,0+)+p_N(0,u,x,0-)}{2}du \\
& = a(x') \int_0^s g_{a(x')u}(x) du - a(x) \int_0^s g_{a(x)u}(x+b(x)u) du \\
& = (a(x')-a(x)) \int_0^s g_{a(x')u}(x) du  + a(x) \int_0^s (g_{a(x')u}(x)   - g_{a(x)u}(x+b(x)u)) du 
\end{align*}

\noindent where we used the exact expression of the transition densities of $(X^{x'}_s)_{s\in [0,h]}$ and $(X^{N,0,x}_{s})_{s\in [0,h]}$ for the last but one equality. From \A{HR} and \A{HE}, one has $|g_{a(x')u}(x) - g_{a(x)u}(x+b(x)u)| \leq C (|b|_{\infty} u^{\frac12}+  |x'-x|^{\eta}) g_{cu}(x)$ which in turn yields
$$
|\E[L^{0}_s(X^{x'})] - \E[L^{0}_s(X^{N})]| \leq C(|b|_{\infty}s^{\frac12} + |x'-x|^{\eta}) s g_{cs}(x).
$$

From the previous computations, we conclude that 
\begin{align}
|\left(\int_{\R} ( p^{x'}(0,s,x,z) - p_N(0,s,x,z) ) (z-x) dz\right) \partial^3_x\tilde{p}(t_{j+1},t_{j'},x,x')| & \leq \frac{C}{(t_{j'}-t_j)^{\frac32}} (|b|_{\infty} s + (2\alpha-1) |x'-x|^{\eta} s g_{cs}(x))  g_{c(t_{j'}-t_j)}(x'-x) \nonumber \\
& \leq C (\frac{s}{(t_{j'}-t_j)^{2-\frac{\eta}{2}}} + (2\alpha-1)\frac{sg_{cs}(x)}{(t_{j'}-t_j)^{\frac{3-\eta}{2}}})g_{c(t_{j'}-t_j)}(x'-x)\label{first:term:rn3}
\end{align}

\noindent where $T \mapsto C:=C(T,b,\sigma)$ is a non-decreasing positive function.

Combining \eqref{second:term:rn3}, \eqref{third:term:rn3} and \eqref{first:term:rn3}, we obtain
\begin{align}
|\mathscr{R}^3_{N}(t_j,t_{j'},x,x') |& \leq C  \left( \int_0^h \left( \frac{s}{(t_{j'}-t_j)^{2-\frac{\eta}{2}}} + (2\alpha-1) \frac{s}{(t_{j'}-t_j)^{\frac{3-\eta}{2}}} g_{cs}(x)\right) ds \right) g_{c(t_{j'}-t_j)}(x'-x) \nonumber \\
& \leq  C\left( \frac{h^2}{(t_{j'}-t_j)^{2-\frac{\eta}{2}}} + (2\alpha-1) \frac{h^2}{(t_{j'}-t_j)^{\frac{3-\eta}{2}}} g_{ch}(x)\right)  g_{c(t_{j'}-t_j)}(x'-x)\label{bound:rn3}
\end{align}

\noindent where $T \mapsto C:=C(T,b,\sigma)$ is a non-decreasing positive function. This last bound completes the proof.
\end{proof}

\begin{lem}\label{Lipschitz:param:series:skew:diff}Under \A{HR} and \A{HE}, there exist constants $C(\lambda,\eta),\, c:=c(\lambda,\eta) \geq1$ such that for all $t\in ]0,T]$, for all $r \geq 0$, for all $(x,x') \in \R \times \R^{*}$, one has
\begin{equation}
\label{bound:parametrix:term:diff}
 \forall s \in [0,t], \ |\tilde{p} \otimes H^{(r)}(0,t+s,x,x') - \tilde{p} \otimes H^{(r)}(0,t,x,x') | \leq  \frac{s}{t} C^{r+1} t^{r \frac{\eta}{2}} \prod_{i=1}^{r-1} B\left(1+\frac{(i-1)\eta}{2},\frac{\eta}{2}\right) g_{c t}(x'-x)
\end{equation}

\noindent with $C:= C(\lambda,\eta)(|b|_{\infty}T^{\frac{1-\eta}{2}}+1)$ and where we use the convention $\prod_{\emptyset} =1$. Consequently, for all $ 0 <s \leq t \leq T$, one has
\begin{equation*}
\forall (x,x')  \times \R \times \R^{*} , \ | p(0,t+s,x,x') - p(0,t,x,x') | \leq C E_{\eta/2,1}(C(|b|_{\infty} T^{\frac12}+T^{\frac{\eta}{2}})) \frac{s}{t} g_{c t}(x'-x).
\end{equation*}

\end{lem}

\begin{proof}
using the fact that $s\in [0,t]$ and standard computations, one has
$$
|\tilde{p}(0,t+s,x,x') - \tilde{p}(0,t,x,x')| \leq \int_t^{t+s} |\partial_v \tilde{p}(0,v,x,x')| dv \leq C(\lambda,\eta) \frac{s}{t}  g_{c t}(x'-x)
$$

\noindent so \eqref{bound:parametrix:term:diff} is valid for $r=0$. Now proceeding by induction we assume that \eqref{bound:parametrix:term:diff} is valid for $r\geq 0$. By a change of variable, one has
\begin{align*}
\tilde{p} \otimes H^{(r+1)}(0,t+s,x,x') - \tilde{p}\otimes H^{(r+1)}(0,t,x,x') & =  \int_0^{t+s} \int_{\R} \tilde{p} \otimes H^{(r)}(0,u,x,z) H(u,t+s,z,x') dz du \\
& - \int_0^{t} \int_{\R} \tilde{p} \otimes H^{(r)}(0,u,x,z) H(u,t,z,x') dz du \\
& = \int_t^{t+s} \int_{\R} \tilde{p}\otimes H^{(r)}(0,t+s-u,x,z) H(0,u,z,x') dz du \\
& + \int_0^t \int_{\R} \left\{ \tilde{p}\otimes H^{(r)}(0,t+s-u,x,z) - \tilde{p}\otimes H^{(r)}(0,t-u,x,z) \right\} H(0,u,z,x') dz du \\
& = I + J.
\end{align*}

For the first term of the above decomposition, from \eqref{iter:step:param} and \eqref{kernel:bound} combined with standard computations, one easily gets
\begin{align*}
| I | & \leq \frac{C^{r+2}}{t^{1-\frac{\eta}{2}}} \left(\int_{t}^{t+s} (t+s-u)^{\frac{r \eta}{2}} du \right) \left( \prod_{i=1}^{r} B\left(1+ \frac{(i-1) \eta}{2}, \frac{\eta}{2}\right) \right)g_{c(t+s)}(x'-x) \\
& \leq \frac{s}{t} C^{r+2} t^{\frac{(r+1)\eta}{2}} \prod_{i=1}^{r} B\left(1+ \frac{(i-1) \eta}{2}, \frac{\eta}{2}\right)  g_{c t}(x'-x)
\end{align*}

\noindent with $C:=C(T,b,\sigma,\eta)= C(\lambda,\eta)(|b|_{\infty}T^{\frac{1-\eta}{2}}+1)$ and where we used that $s\in [0,t]$ for the last inequality. For the second term, we write $J = J_1 + J_2$ with
\begin{align*}
 J_1 & = \int_{0}^{t/2} \int_{\R} \left\{ \tilde{p}\otimes H^{(r)}(0,t+s-u,x,z) - \tilde{p}\otimes H^{(r)}(0,t-u,x,z)\right\}  H(0,u,z,x') dz du \\
 J_2 &  = \int_{t/2}^{t} \int_{\R} \left\{ \tilde{p}\otimes H^{(r)}(0,t+s-u,x,z) - \tilde{p}\otimes H^{(r)}(0,t-u,x,z)\right\}  H(0,u,z,x') dz du .
\end{align*}

First assume that $s\in [0,t/2]$. Then one has $s\in [0,t-u]$ for all $u \in [0,t/2]$ so that from the induction hypothesis and \eqref{kernel:bound}, one gets
\begin{align*}
|J_1| & \leq C^{r+2} \frac{s}{t^{1-\frac{\eta}{2}}} \left( \int_0^{t/2} (t-u)^{\frac{(r-1)}{2}\eta} u^{-(1-\frac{\eta}{2})} du \right) \left(\prod_{i=1}^{r-1} B\left(1 + \frac{(i-1)\eta}{2},\frac{\eta}{2} \right) \right) g_{c t}(x'-x) \\
& \leq \frac{s}{t} C^{r+2} t^{\frac{(r+1)}{2}\eta} \prod_{i=1}^{r} B\left(1+ \frac{(i-1)\eta}{2},\frac{\eta}{2}\right) g_{c t }(x'-x)
\end{align*}

\noindent with $C:=C(T,b,\sigma,\eta)= C(\lambda,\eta)(|b|_{\infty}T^{\frac{1-\eta}{2}}+1)$. Now, if $s \in (t/2,t]$, one writes $J_1= J^{1}_1 + J^{2}_1$ with
\begin{align*}
J^{1}_1 & = \int_0^{t/2} \int_{\R}  \left\{ \tilde{p}\otimes H^{(r)}(0,t-u+\frac{t}{2} + (s-\frac{t}{2}),x,z) - \tilde{p}\otimes H^{(r)}(0,t-u + \frac{t}{2},x,z)\right\}  H(0,u,z,x') dz du, \\
J^{2}_1 & = \int_0^{t/2} \int_{\R}  \left\{ \tilde{p}\otimes H^{(r)}(0,t-u+\frac{t}{2} ,x,z) - \tilde{p}\otimes H^{(r)}(0,t-u,x,z)\right\}  H(0,u,z,x') dz du. 
\end{align*}

From the induction hypothesis and \eqref{kernel:bound} using the fact that $t/2 \leq t-u$ for $u\in [0,t/2]$, one obtains
\begin{align*}
|J^{1}_1| & \leq C^{r+2}  \left( \int_0^{t/2} \frac{(s-\frac{t}{2})}{\frac{t}{2}+(t-u)} (t-u + \frac{t}{2})^{\frac{r}{2}\eta} u^{-(1-\frac{\eta}{2})} du \right) \left(\prod_{i=1}^{r-1} B\left(1 + \frac{(i-1)\eta}{2},\frac{\eta}{2} \right) \right) g_{c t}(x'-x) \\
& \leq \frac{s}{t} C^{r+2} t^{\frac{(r+1)}{2}\eta} \prod_{i=1}^{r} B\left(1+ \frac{(i-1)\eta}{2},\frac{\eta}{2}\right) g_{c t }(x'-x).
\end{align*}

Using similar arguments with $s\geq t/2$, one also gets
\begin{align*}
|J^{2}_1| & \leq \frac{s}{t} C^{r+2} t^{\frac{(r+1)}{2}\eta} \prod_{i=1}^{r} B\left(1+ \frac{(i-1)\eta}{2},\frac{\eta}{2}\right) g_{c t }(x'-x)
\end{align*}

\noindent which finally yields
$$
|J_1| \leq \frac{s}{t} C^{r+2} t^{\frac{(r+1)}{2}\eta} \prod_{i=1}^{r} B\left(1+ \frac{(i-1)\eta}{2},\frac{\eta}{2}\right) g_{c t }(x'-x).
$$

The last term $J_2$ is given by the sum of three terms, namely
\begin{align*}
J^{1}_2 &= -\int_0^{s} \int_{\R} \tilde{p} \otimes H^{(r)}(0,u,x,z) H(0,t-u+s,z,x') dz du,  \\
J^{2}_2 & = \int_{\frac{t}{2}}^{\frac{t}{2}+s} \int_{\R} \tilde{p}\otimes H^{(r)}(0,u,x,z) H(0,t-u+s,z,x') dz du, \\
J^{3}_2 & = \int_{0}^{\frac{t}{2}} \int_{\R} \tilde{p}\otimes H^{(r)}(0,u,x,z) \left\{ H(0,t-u+s,z,x') - H(0,t-u,z,x') \right\} dz du.
\end{align*}

Using \eqref{iter:step:param} and \eqref{kernel:bound} with similar computations, one easily gets
\begin{align*}
|J^{1}_2| &\leq  C^{r+2} \left( \int_0^s u^{\frac{r \eta}{2}} \frac{1}{(t+s-u)^{1-\frac{\eta}{2}}} du \right) \prod_{i=1}^{r} B\left(1+ \frac{(i-1)\eta}{2},\frac{\eta}{2}\right) g_{c t}(x'-x) \\
& \leq \frac{s}{t} C^{r+2} t^{\frac{(r+1)\eta}{2}} \prod_{i=1}^{r} B\left(1+ \frac{(i-1)\eta}{2},\frac{\eta}{2}\right) g_{c t}(x'-x)
\end{align*}

\noindent and similarly
\begin{align*}
|J^{2}_2| &\leq  C^{r+2}\frac{1}{t^{1-\frac{\eta}{2}}} \left( \int_{\frac{t}{2}}^{\frac{t}{2}+s} u^{\frac{r \eta}{2}}du \right) \prod_{i=1}^{r} B\left(1+ \frac{(i-1)\eta}{2},\frac{\eta}{2}\right) g_{c t}(x'-x) \\
& \leq \frac{s}{t} C^{r+2}  t^{\frac{(r+1)\eta}{2}} \prod_{i=1}^{r} B\left(1+ \frac{(i-1)\eta}{2},\frac{\eta}{2}\right) g_{c t}(x'-x).
\end{align*}

From \A{HR} and \A{HE}, one has $| \partial_v H(0,v, z,x')| \leq C v^{-(2-\frac{\eta}{2})} g_{c v}(x'-z)$, $v\in (0,T]$ and using the fact that $s\in [0,t]$ we derive
$$
 | H(0,t-u+s,z,x') - H(0,t-u,z,x') | \leq C \int_{t-u}^{t-u +s } | \partial_v H(0,v,z,x')| dv \leq C \frac{s}{(t-u)^{2- \frac{\eta}{2}}} g_{c (t+s-u)}(x'-z)
$$

\noindent which in turn yields
\begin{align*}
|J^{3}_2| &\leq \frac{s}{t^{1-\frac{\eta}{2}}}  C^{r+2} \left( \int_{0}^{\frac{t}{2}} u^{\frac{r \eta}{2}} \frac{1}{(t-u)^{1-\frac{\eta}{2}}}du \right) \prod_{i=1}^{r} B\left(1+ \frac{(i-1)\eta}{2},\frac{\eta}{2}\right) g_{c (t+s)}(x'-x) \\
& \leq \frac{s}{t}  C^{r+2} t^{\frac{(r+1)\eta}{2}}\prod_{i=1}^{r} B\left(1+ \frac{(i-1)\eta}{2},\frac{\eta}{2}\right) g_{c t}(x'-x).
\end{align*}

This completes the proof of \eqref{bound:parametrix:term:diff}. Summing \eqref{bound:parametrix:term:diff} from $r=0$ to infinity and using the asymptotics of the Gamma function yield
$$
\forall (t,x,x') \in ]0,T] \times \R \times \R^{*} , \ | p(0,t+s,x,x') - p(0,t,x,x') | \leq C \frac{s}{t} g_{c t}(x'-x)
$$ 

\noindent for some constants $C:=C(T,b,\sigma)= C(\lambda,\eta)(|b|_{\infty}T^{\frac{1-\eta}{2}}+1) E_{\eta/2,1}(C(|b|_{\infty} T^{\frac12}+T^{\frac{\eta}{2}}))$, $C(\lambda,\eta),  c:=c(\lambda,\eta)>1$.
\end{proof}

\providecommand{\bysame}{\leavevmode\hbox to3em{\hrulefill}\thinspace}
\providecommand{\MR}{\relax\ifhmode\unskip\space\fi MR }
\providecommand{\MRhref}[2]{%
  \href{http://www.ams.org/mathscinet-getitem?mr=#1}{#2}
}
\providecommand{\href}[2]{#2}


\bibliographystyle{alpha}
\bibliography{bibli}

\end{document}

%% file: macro.tex
\newcommand \A[1]{{\bf (#1)}}
\def\B{{\cal B}}

\def\B{{\cal B}}

\def\P{{\mathbb{P}}  }

\def\bint#1^#2{\displaystyle{\int_{#1}^{#2}}}
\def\bsum#1^#2{\displaystyle{\sum_{#1}^{#2}}}

\def\xdt_#1{X_#1(\Delta t)}

\newtheorem{THM}{Theorem}[section]

\newtheorem{PROP}{Proposition}[section]

\newtheorem{REM}{Remark}[section]